\documentclass[times]{article}

\usepackage[margin=3.3cm]{geometry}

\usepackage{amsmath}
\usepackage{amssymb}
\usepackage{tikz,pgfplots}
\usepackage{subcaption}
\usepackage{multirow}
\usepackage{enumerate}
\usepackage{enumitem}
\usepackage{stmaryrd}

\usepackage[colorlinks,bookmarksopen,bookmarksnumbered,citecolor=red,urlcolor=red]{hyperref}
\usepackage{amsthm}
\theoremstyle{plain}
\newtheorem{theorem}{Theorem}

\newtheorem{lemma}{Lemma}
\newtheorem{definition}{Definition}

\newtheorem{remark}{Remark} 
\newtheorem{assumption}{Assumption}

\newcommand{\matid}{\ensuremath{\mathbf{{I}}}}
\newcommand{\bu}{\ensuremath{\mathbf{u}}}
\newcommand{\by}{\ensuremath{\mathbf{y}}}
\newcommand{\bR}{\ensuremath{\mathbf{R}}}

\newcommand{\bRst}{\ensuremath{{\mathbf{R}^s}^\top}}

\newcommand{\bB}{\ensuremath{\mathbf{B}}}
\newcommand{\bD}{\ensuremath{\mathbf{D}}}
\newcommand{\bM}{\ensuremath{\mathbf{M}}}
\newcommand{\bW}{\ensuremath{\mathbf{W}}}

\newcommand{\bV}{\ensuremath{\mathbf{V}}}
\newcommand{\bY}{\ensuremath{\mathbf{Y}}}

\newcommand{\bPi}{\ensuremath{\boldsymbol{\Pi}}}

\newcommand{\bL}{\ensuremath{\boldsymbol{\Lambda}}}
\newcommand{\bv}{\ensuremath{\mathbf{v}}}
\newcommand{\bx}{\ensuremath{\mathbf{x}}}
\newcommand{\bb}{\ensuremath{\mathbf{b}}}

\newcommand{\bA}{\ensuremath{\mathbf{A}}}
\newcommand{\bH}{\ensuremath{\mathbf{H}}}
\newcommand{\bN}{\ensuremath{\mathbf{N}}}
\newcommand{\bS}{\ensuremath{\mathbf{S}}}

\newcommand{\bz}{\ensuremath{\mathbf{z}}}

\newcommand{\AS}{\ensuremath{^{\mathrm{AS}}}}
\newcommand{\NN}{\ensuremath{^{\mathrm{NN}}}}
\newcommand{\uAS}{\ensuremath{_{\mathrm{AS}}}} 
\newcommand{\uASp}{\ensuremath{_{\mathrm{AS},+}}} 
\newcommand{\uNN}{\ensuremath{_{\mathrm{NN}}}}
\newcommand{\bHtwo}{\ensuremath{\bH_2}} 
\newcommand{\st}{\ensuremath{^{t}}}
\newcommand{\0}{\ensuremath{^{0}}}
\newcommand{\s}{\ensuremath{^{s}}}

\newcommand{\range}{\ensuremath{\operatorname{range}}}

\newcommand{\bphi}{\ensuremath{\boldsymbol{\phi}}}

\newcommand{\soneN}{\ensuremath{s\in\llbracket 1, N \rrbracket}}

\usetikzlibrary{fit,calc}
\usepackage{xcolor}
\usepackage[nofill,customcolors, markings]{hf-tikz}
\hfsetbordercolor{white}
\usetikzlibrary {positioning}

\usepackage{stmaryrd}

\graphicspath{{Figures.d/}}

\title{Fully algebraic domain decomposition preconditioners with adaptive spectral bounds
}

\author{Lo\"ic Gouarin\footnotemark[2] \and Nicole Spillane\footnotemark[3]}

\begin{document}

\maketitle

\renewcommand{\thefootnote}{\fnsymbol{footnote}}

\footnotetext[2]{CNRS, CMAP, \'Ecole Polytechnique, Institut Polytechnique de Paris, 91128 Palaiseau Cedex, France (\textit{loic.gouarin@cmap.polytechnique.fr})} 
\footnotetext[3]{CNRS, CMAP, \'Ecole Polytechnique, Institut Polytechnique de Paris, 91128 Palaiseau Cedex, France (\textit{nicole.spillane@cmap.polytechnique.fr})} 

\begin{abstract}
In this article a new family of preconditioners is introduced for symmetric positive definite linear systems. The new preconditioners, called the AWG preconditioners (for Algebraic-Woodbury-GenEO) are constructed algebraically. By this, we mean that only the knowledge of the matrix $\bA$ for which the linear system is being solved is required. Thanks to the GenEO spectral coarse space technique, the condition number of the preconditioned operator is bounded theoretically from above. This upper bound can be made smaller by enriching the coarse space with more spectral modes. 

The novelty is that, unlike in previous work on the GenEO coarse spaces, no knowledge of a partially non-assembled form of $\bA$ is required. Indeed, the spectral coarse space technique is not applied directly to $\bA$ but to a low-rank modification of $\bA$ of which a suitable non-assembled form is known by construction. The extra cost is a second (and to this day rather expensive) coarse solve in the preconditioner. One of the AWG preconditioners has already been presented in the short preprint \cite{spillane:hal-03187092}. This article is the first full presentation of the larger family of AWG preconditioners. It includes proofs of the spectral bounds as well as numerical illustrations. 
\end{abstract}
%
\pagestyle{myheadings}
\markboth{Lo\"ic GOUARIN, Nicole SPILLANE}{L. Gouarin and N. Spillane. Fully algebraic DD with adaptive spectral bounds}
\textbf{Keywords:}
preconditioner, domain decomposition, coarse space, algebraic, linear system, Woodbury matrix identity
%

\noindent\textbf{AMS classification:}
65F10, 65N30, 65N55

\tableofcontents

%

\section{Introduction}

Throughout this article we consider the problem of finding $\bx_* \in \mathbb R^n$ that is the solution of the following linear system
\begin{equation}
\bA \bx_* = \mathbf{b}, \text{ where $\bA \in \mathbb R^{n\times n} $ is symmetric positive and definite (spd),} 
\label{eq:Ax=b}
\end{equation}
for a given right hand side $\mathbf{b} \in \mathbb R^n$. 

The applications we have in mind are ones for which $\bA$ is sparse and the number $n$ of unknowns is very large. Hence, we study parallel solvers and more specifically preconditioners for the preconditioned conjugate gradient (PCG) method \cite{saad2003iterative}[Section 9.2]. Our objective is to propose a new preconditioner $\bH$ for solving \eqref{eq:Ax=b} in such a way that the condition number of $\bH \bA$ is bounded from above by a \textit{small enough} constant chosen by the user. This guarantees that the PCG method will converge in \textit{sufficiently few} iterations \cite{ToselliWidlund_book2005}[Lemma C.10]. Two-level domain decomposition preconditioners already exist that satisfy such a nice property (specifically the spectral coarse space methods described below). These methods, however, rely on the knowledge of some partially unassembled form of matrix $\bA$. The additional challenge that we tackle in this work is that the new preconditioner must rely only on the knowledge of the matrix $\bA$, and this is the meaning of the word \textit{algebraic} as it is in Algebraic Multigrid (AMG) \cite{brandt1984algebraic,ruge1987algebraic,xu2017algebraic}. 

Very generally speaking, domain decomposition methods partition the domain $\Omega$ in which the solution is sought into smaller spaces $\Omega\s$, indexed by $s\in \llbracket 1,N \rrbracket$, and characterized by prolongation matrices ${\bR\s}^\top$ that satisfy $\sum_{s=1}^N \operatorname{range}({\bR\s}^\top) = \Omega$. One-level domain decomposition preconditioners then approximate $\bA^{-1}$ by a sum (interpolated by the ${\bR\s}^\top$) of inverses of well-chosen problems $\tilde\bA\s$. Two-level domain decomposition methods have an additional space called the coarse space (generated by the columns of a matrix ${\bR\0}^\top$). A coarse problem $\tilde \bA \0$ is solved in the coarse space. As an example, an application of the two-level additive preconditioner to a vector $\bx \in \mathbb R^n$ takes the form:  $\sum_{s=1}^N \by\s$ where 
\[
  \by\s = {
    \tikzmarkin[opacity=0]{Rt}\bR\s\tikzmarkend{R}
  }^\top
  (\tikzmarkin[opacity=0.0]{A}\tilde\bA\s\tikzmarkend{A})\,^{-1}
  \tikzmarkin[opacity=0.0]{R}\bR\s\tikzmarkend{R} \bx \quad  \text{ and } \quad
  \by\0 =
  \tikzmarkin[opacity=0.0]{Rt0}{\bR\0}^\top\tikzmarkend{Rt0}
  (\tikzmarkin[opacity=0.0]{A0}\tilde\bA\0\tikzmarkend{A0})\,^{-1}
  \tikzmarkin[opacity=0.0]{R0}\bR\0\tikzmarkend{R0} \bx.
\]
\begin{tikzpicture}[remember picture, overlay]\footnotesize
  \coordinate (RtB) at ($(Rt)+(0.8em,-1.5em)$);
  \coordinate (RtA) at ($(RtB)+(0,-0.3cm)$);
  \node[left=0.2cm of RtA.north west] (textRt) at (RtA) {prolongation (by 0) to $\Omega$};
  \draw[->] (textRt)-|(RtB);
  \coordinate (AB) at ($(A)+(0.8em,-1.5em)$);
  \coordinate (AA) at ($(AB)+(0,-0.6cm)$);
  \node[left=0.2cm of AA.north west] (textA) at (AA) {local solve};
  \draw[->] (textA)-|(AB);
  \coordinate (RB) at ($(R)+(0.8em,-1.5em)$);
  \coordinate (RA) at ($(RB)+(0,-0.9cm)$);
  \node[left=0.2cm of RA.north west] (textR) at (RA) {restriction to $\Omega\s$};
  \draw[->] (textR)-|(RB);
  \coordinate (Rt0B) at ($(Rt0)+(0.8em,-1.5em)$);
  \coordinate (Rt0A) at ($(Rt0B)+(0,-0.9cm)$);
  \node[right=0.2cm of Rt0A.north west] (textRt0) at (Rt0A) {interpolation back into $\Omega$};
  \draw[->] (textRt0)-|(Rt0B);
  \coordinate (A0B) at ($(A0)+(0.8em,-1.5em)$);
  \coordinate (A0A) at ($(A0B)+(0,-0.6cm)$);
  \node[right=0.2cm of A0A.north west] (textA0) at (A0A) {coarse solve};
  \draw[->] (textA0)-|(A0B);
  \coordinate (R0B) at ($(R0)+(0.8em,-1.5em)$);
  \coordinate (R0A) at ($(R0B)+(0,-0.3cm)$);
  \node[right=0.2cm of R0A.north west] (textR0) at (R0A) {coarse interpolation};
  \draw[->] (textR0)-|(R0B);

\end{tikzpicture}
\vspace{0.5cm}

The notation in the previous line should seem completely natural to readers who have already studied domain decomposition. The coarse contribution $\by\0$ is written separately to insist on the fact that it plays a different role to the other $\by\s$. Usually, the coarse space is computed for a given choice or the $\bR\s$ and $\tilde \bA\s$. 

The choice of the coarse space is a very crucial topic in domain decomposition. 
Over the last decade, the range of symmetric positive definite problems for which two-level domain decomposition preconditioners can be made scalable and robust by a good choice of the coarse space has become very large by the development of so-called spectral coarse spaces. The following list gives an overview of some of these contributions: \cite{Nataf:TDD:2010,Nataf:CSC:2011,Dolean:ATL:2011,2010GalvisJ_EfendievY-a0,2010GalvisJ_EfendievY-ab,efendiev2012robust,2011SpillaneCR,spillane2013abstract,gander2017shem,agullo2019robust,heinlein2019adaptive} for Additive Schwarz, \cite{yu2020additive} for Additive Average Schwarz, \cite{SPILLANE:2013:FETI_GenEO_IJNME,agullo2019robust} for BDD and FETI, \cite{klawonn2016comparison,klawonn2016adaptive,calvo2016adaptive,kim2017bddc,da2017adaptive,pechstein2017unified,klawonn2015toward,zampini2016pcbddc} for BDDC and/or FETI-DP, \cite{haferssas2017additive} for Optimized Schwarz and \cite{marchand2020two} in the context of boundary element methods.

In this article it is particularly referred to the GenEO coarse spaces \cite{2011SpillaneCR,spillane2013abstract,SPILLANE:2013:FETI_GenEO_IJNME} to which one of the authors contributed. The abstract theory of coarse spaces of the GenEO family in~\cite{AbstractGenEO} is applied within the definition and analysis of the new Algebraic-Woodbury-GenEO (AWG) preconditioners. 

The spectral coarse spaces mentioned above, and in particular the GenEO coarse spaces, are computed by partially solving one or two generalized eigenvalue problems in each subdomain, then selecting either the lowest or highest-frequency eigenvectors and prolongating them to the global domain. To the best of the authors' knowledge, all of the spectral coarse spaces for which there are no assumptions on the shapes of subdomains and the distribution of material parameters require the knowledge of a set of symmetric positive semi-definite (spsd) matrices $\bN\s$ that satisfy 
\begin{equation}
\label{eq:spsd-splitting}
\exists \, C> 0 \text{ such that } \sum_{s=1}^N \langle \bx, {\bR\s}^\top \bN\s \bR\s \bx \rangle \leq C  \langle \bx, \bA \bx \rangle  \quad \forall \, \bx \in \mathbb R^n. 
\end{equation}
The matrices $\bN\s$ enter directly into the coarse space construction \textit{via} the choice of matrix pencil for the generalized eigenvalue problems. In other words they play an essential role.

For matrices arising from discretized PDEs the above assumption is far from unnatural and the matrices $\bN\s$ are not expensive to compute as long as it is known at assembly that they are required. Indeed, as an example, if $\bA$ arises from the finite element discretization (with basis functions  $\{\phi_k\}_{k=1,\dots,n}$) of the Laplace equation over some domain $\omega \subset \mathbb R^{2 \text{ or } 3}$, then the coefficients in $\bA$ are $\int_{\Omega} \nabla \phi_i \cdot \nabla \phi_j $. Assuming that the degrees of freedom selected by the restriction matrix $\bR\s$ are those in some $\omega\s \subset \omega$, $\bN\s$ can be taken to be the matrix whose coefficients are $ \int_{\omega\s} \nabla \phi_i \cdot \nabla \phi_j $ (for the basis functions $\phi_i$ and $\phi_j$ whose support intersects $\omega\s$). This is how condition \eqref{eq:spsd-splitting} is usually satisfied. These $\bN\s$ are sometimes called the local Neumann matrices. Partial assembly over subdomains is neither hard conceptually nor expensive computationally and the purpose of this article is not to rule it out when it is possible. There are however many cases where only $\bA$ is known or available without writing many more lines of code. Then, the unassembled information is simply lost and the GenEO coarse spaces can't be computed. This is quite a common scenario: the problem may have been assembled by another user or with another piece of software. In this case only black box algorithms can be used. 

Direct solvers, like MUMPS \cite{MUMPS:1,MUMPS:2}, belong to the category of black box solvers and they are the most efficient up to a certain problem size. In the field of domain decomposition, the authors of \cite{li2017low} propose an algebraic preconditioner under the name DD-LR (for Domain Decomposition based Low-Rank). The original matrix is rewritten in a particular form inspired by domain decomposition. The inverse of $\bA$ can then be expressed in terms of the components in that formulation thanks to the Woodbury matrix identity. Finally a low rank approximation of one of the terms is performed in order to get an approximation of $\bA^{-1}$ that can serve as a preconditioner. Our AWG preconditioners also exploit the Woodbury matrix identity but the modification of $\bA$ that it is applied to is entirely different.
 Also closely related to domain decomposition are the multigrid algorithms, a very well-established set of solvers that solves the problem by iterating over the original (fine) problem as well as coarser and coarser approximations of it. The original multigrid algorithm \cite{brandt1977multi,briggs2000multigrid} is often referred to as geometric multigrid as it requires more information about the problem than just the matrix $\bA$. Our objective with this algorithm is not unlike the objective of Algebraic Multigrid (AMG) which is to make multigrid applicable in cases where less information about the problem is accessible or some assumptions are not satisfied (see \cite{brandt1984algebraic,ruge1987algebraic} for the original contributions or \cite{xu2017algebraic} for a unified presentation and theory of many multigrid algorithms). Our ambition here is to propose an algorithm as easy to apply as are the Algebraic Multigrid algorithms.

In this article, a new preconditioner (with several variants) is proposed for the cases where $\bA$ is already assembled. It is a domain decomposition preconditioner with two coarse space. The methodology is the following:
\begin{itemize}
\item The problem matrix $\bA$ is split into symmetric, but possibly indefinite, matrices $\bB\s$ as $\bA = \sum_{s=1}^N {\bR\s}^\top \bB\s \bR\s$.
\item The positive parts $\bA_+^s$ of the matrices $\bB\s$ are computed and assembled to form a global matrix $\bA_+ = \sum_{s=1}^N {\bR\s}^\top \bA\s_+ \bR\s$. By construction, a splitting of $\bA_+$ into spsd matrices (\textit{i.e.}, a suitable partially unassembled form) of $\bA_+$ is known. In other words, with $\bN\s = \bA_+^s$ and $C=1$, \eqref{eq:spsd-splitting} is satisfied. Consequently, two-level preconditioners with GenEO coarse spaces can be computed for $\bA_+$ by applying the abstract theory in \cite{AbstractGenEO}. 
\item Finally, the Woodbury matrix identity relates the inverses of $\bA$ and of $\bA_+$ and makes apparent that a good preconditioner for $\bA$ can be obtained by adding a second coarse space to a GenEO preconditioner for $\bA_+$.
\end{itemize}
Full theory for the condition number of the new preconditioned operators is given as well as numerical illustrations. The outline of the remainder of this article is as follows. In Section~\ref{sec:AbstractSchwarz}, some elements of the Abstract Schwarz theory \cite{ToselliWidlund_book2005} are recalled in their fully algebraic form. For readers less familiar with domain decomposition, the general form of a one-level and a two-level domain decomposition preconditioner is given. In Section~\ref{sec:Apos}, the new operator $\bA_+$ is introduced and four preconditioners with their GenEO coarse spaces are considered for $\bA_+$. For each one, the spectral bounds are given by applying a result from \cite{AbstractGenEO}. Then, in Section~\ref{sec:newprecs}, $\bA$ is viewed as a modification of $\bA_+$ and the Woodbury matrix identity is applied. This makes apparent how to add a second coarse space to the preconditioners for $\bA_+$ in order to get a preconditioner for $\bA$ that satisfies nice convergence bounds. Each of these new preconditioners is indexed by one or two parameters (or thresholds) that can be adjusted to decrease the condition number of the preconditioned operator by enriching the GenEO coarse space with more spectral modes. Some comments are also made about the implementation of the new preconditioners. Finally, Section~\ref{sec:Numerical} presents some numerical results with the objective of confirming the theoretical results and illustrating the practical behaviour of the new AWG preconditioners. 

\section{Abstract Schwarz Framework in the Algebraic Setting}
\label{sec:AbstractSchwarz}

An algebraic version of the abstract Schwarz framework is introduced in this section. This means that all the domain decomposition-type operators are written only in terms of vectors in $\mathbb R^n$.

\subsection{Subdomains}

Let $\Omega = \llbracket 1, n \rrbracket$ be the set of all indices in $\mathbb R^n$. 

\begin{definition}[Partition of $\Omega$] 
\label{def:partition}
A set $\left( \Omega\s \right)_{s=1,\dots,N}$ of $N\in \mathbb N$ subsets of $\Omega = \llbracket 1, n \rrbracket$ is called a partition of $\Omega$ if
\[
\Omega = \bigcup_{s=1}^N \Omega\s. 
\]
The partition is said to have at least minimal overlap if Assumption~\ref{ass:minoverlap} is satisfied.
\end{definition}

\begin{assumption}[Minimal overlap]
\label{ass:minoverlap}
For any pair of indices $(i,j) \in \llbracket 1, n \rrbracket^2$, denoting by $A_{ij}$ the coefficient of $\bA$ at the $i$-th line and $j$-th column, 
\[
A_{ij} \neq 0 \Rightarrow \left( \exists \, s \in \llbracket 1, N \rrbracket \text{ such that } \{i,j\} \subset \Omega\s \right). 
\]
\end{assumption}

The usual global-to-local restriction matrices are defined next.

\begin{definition}
\label{def:Rs}
For each $s=1,\dots,N$, let $n \s$ be the cardinality of $\Omega\s$. Then, let $\bR\s \in \mathbb R^{n\s \times n}$ be the restriction matrix defined by: $\bR\s$ is zero everywhere except for the block formed by the columns in $\Omega\s$ which is the $n\s\times n\s$ identity matrix. 
\end{definition}

By simply performing the multiplications it can be proved that
\[
{\bR\s}^\top \bR\s \text{ is diagonal and that } {\bR\s} {\bR\s}^\top = \matid \text{ (the identity matrix in $ \mathbb R^{n\s} $ )}. 
\]

\subsection{Partition of unity, coloring constant}

In the construction and analysis of the preconditioners, two more elements from the abstract Schwarz theory are needed: the partition of unity matrices and the coloring constants.  

\begin{assumption}[Partition of unity matrices]
\label{ass:Ds}
Let $\{\bD \s \in \mathbb R^{n\s \times n\s}; \,s= 1, \dots, N\}$ be a family of matrices that satisfies 
\begin{equation}
\label{eq:POU}
\matid = \sum_{s=1}^N \bR \s\, ^\top \bD\s \bR \s, \text{ with $\matid$ the $n \times n$ identity matrix}.
\end{equation}
\end{assumption}

One way of fulfilling Assumption~\ref{ass:Ds} is to choose the following partition of unity matrices. 

\begin{definition}[Possible choice of partition of unity $\bD\s$]
\label{lem:defofDs}
First, let $\bD \in \mathbb R^{n\times n}$ be the non-singular diagonal matrix defined by
\[
\bD := \left(\sum_{t=1}^N {\bR \st}^\top  \bR \st\right)^{-1}. 
\]
Then, for each $\soneN$, let $\bD\s \in \mathbb R^{n\s \times n\s}$ be defined by  
\[
\bD \s :=  \bR \s \bD {\bR \s}^\top.
\]
\end{definition}

The matrices in the above definition satisfy Assumption~\ref{ass:Ds} (see \cite{AbstractGenEO}[Lemma 4]). Their coefficients are the inverses of the multiplicity of each degree of freedom. Next, the coloring constant is defined in agreement with \cite{ToselliWidlund_book2005}[Section 2.5.1]. The dependency of the coloring constant on the matrix with respect to which orthogonality is taken is written explicitly. 

\begin{definition}[Coloring constant]
\label{def:color}
Let $\bM \in \mathbb R^{n\times n}$ be a symmetric matrix. Let $\mathcal N(\bM) \in \mathbb N$ be such that there exists a set $\{ \mathcal C_j; \,  1 \leq j \leq \mathcal N(\bM)\}$ of pairwise disjoint subsets of $\llbracket 1, N \rrbracket$ satisfying
\[
\llbracket 1,N\rrbracket = \bigcup_{1\leq j \leq \mathcal N(\bM)} \mathcal C_j\quad \text{and}\quad \forall j \in \llbracket 1, \mathcal N(\bM) \rrbracket:\,  \{s,t\} \subset \mathcal C_j \Rightarrow (\bR\s \bM {\bR \st}^\top = \mathbf{0} \text{ or  } s=t).
\]
\end{definition}

\subsection{Abstract Schwarz preconditioners} 
\label{subs:abstract-prec}

One-level abstract Schwarz preconditioners are of the form:
\begin{equation}
\label{eq:bH}
\bH := \sum_{s=1}^N {\bR\s}^\top {\tilde \bA\s}\,^\dagger \bR\s.
\end{equation}
where for each $s=1,\dots,N$, it is assumed that 
\[
\tilde \bA\s \in  \mathbb R^{n\s \times n\s} \text{is an spsd matrix, and that } {\tilde \bA\s}\,^\dagger \text{ is the pseudo-inverse of $\tilde \bA \s$}. 
\]

Two-level domain decomposition preconditioners have two extra ingredients compared to the one-level method that they are based on: a coarse space and a coarse solver. Let's assume that the coarse space is denoted by $V\0$ and that the interpolation operator $\bR\0$ satisfies Assumption~\ref{ass:V0}.  

\begin{assumption}
\label{ass:V0}
A basis for the coarse space $V\0$ is stored in the lines of a matrix denoted ${\bR\0}$:
\[
V\0 = \range({\bR\0}^\top); \quad {\bR\0} \in \mathbb R^{n\0 \times n};\quad n^0 = \dim (V\0);\quad n\0 < n.
\]
\end{assumption}

The most common choice for the coarse solver, and the one that we wish to introduce, is the so called \textit{exact} solver, where the word \textit{exact} is with respect to the problem being solved. If the matrix in the linear system is an spd matrix $\tilde \bA$ then the matrix that is inverted during the coarse solve is $\bR\0 \tilde \bA {\bR\0}^\top$.

Even with the same $\tilde \bA\s$ and $\bR\0$, there are still at least two two-level preconditioners with exact coarse spaces: the two-level additive preconditioner (denoted by $\bH_{\mathrm{ad}}  $), and the hybrid preconditioner (denoted by $\bH_{\mathrm{hyb}}$ and also called the deflated preconditioner). They are defined as follows:  
\begin{equation}
\bH_{\mathrm{ad}} := \bH + {\bR\0}^\top (\bR\0 \tilde \bA {\bR\0}^\top)^{-1} \bR\0,
\label{eq:defH-ad}
\end{equation}
and 
\[
\bH_{\mathrm{hyb}} := \bPi \bH \bPi^\top + {\bR\0}^\top (\bR\0 \tilde\bA {\bR\0}^\top)^{-1} \bR\0;
\]
where
\begin{equation}
\label{eq:bPi}
\bPi := \matid - {\bR\0}^\top (\bR\0 \tilde\bA {\bR\0}^\top)^{-1} \bR\0 \tilde\bA; \quad \matid \text{ is the $n \times n$ identity matrix}.
\end{equation} 

The generic notation $\tilde \bA$ has very deliberately been used in the previous equations instead of $\bA$. Indeed, the two-level preconditioner in this article is a preconditioner for a new matrix that will be denoted $\bA_+$. The next section gives the definition of $\bA_+$ and the choices of $\tilde \bA\s$and $V\0$ that make the characterization of an abstract two-level preconditioner for $\bA_+$ complete.

\section{A new matrix $\bA_+$ and its two-level GenEO preconditioners} 
\label{sec:Apos}

This section introduces a lot of the new operators and notation. Since it is not known how to algebraically find local spsd matrices that satisfy \eqref{eq:spsd-splitting}, it is chosen to relax the assumption by allowing the matrices to be indefinite. Precisely, in subsection~\ref{subs:def-of-Bs} it is assumed that symmetric matrices $\bB\s$ are known such that $\bA = \sum_{s=1}^N {\bR\s}^\top \bB\s \bR\s$ and one possible (algebraic) choice of $\bB\s$ is given. In Section~\ref{subs:def-of-A+}, each $\bB\s$ is in turn split into an spsd part and a symmetric negative semi-definite part as $\bB\s = \bA_+\s - \bA_-\s$. Finally, the spsd parts are assembled to form $\bA_+ = \sum_{s=1}^N {\bR\s}^\top \bA_+\s \bR\s$. The resulting matrix $\bA_+$ is shown to be symmetric positive definite and, by construction, it satisfies \eqref{eq:spsd-splitting} with $\bN\s = \bA_+\s$ and $C = 1$. In other words, a very nice characteristic of $\bA_+$ is that the abstract GenEO theory applies to defining and analyzing two-level preconditioners for $\bA_+$. Four of these are considered in Section~\ref{subs:def-H2}.

\subsection{A splitting of $\bA$ into symmetric matrices}
\label{subs:def-of-Bs}

The matrices $\bB\s$ in the assumption below are the starting point for the new preconditioners. To make the construction complete, an example of such matrices, is given below. It is (of course) constructed algebraically and is the one used in our numerical computations. This choice is however far from unique. 

\begin{assumption}
\label{ass:symsplit}
Let's assume that there exists a family of symmetric matrices $\bB\s \in \mathbb R^{n\s \times n\s}$ for $s=1,\dots,N$ such that
\[
\bA = \sum_{s=1}^N {\bR\s}^\top \bB\s \bR\s.
\]
\end{assumption}

Such a family of matrices can always be chosen under Assumption~\ref{ass:minoverlap} (minimal overlap). Indeed, one possible choice is given in the next definition.

\begin{definition}[Possible choice of matrice $\bB\s$]
\label{def:Bs}
Let $\bS(\bA)$ be the $n \times n$ boolean matrix that shares the same sparsity pattern as $\bA$:
\[
(S(\bA))_{ij} := 
\left\{ \begin{array}{l}
1 \text{ if } A_{ij} \neq 0 \\
0 \text{ otherwise}
\end{array}\right. 
\text{ for any } i,\,j \in \llbracket 1, n \rrbracket.
\] 
Then let $\bM_\mu$ be the matrix that counts the number of subdomains that each pair of indices in $\{ \{i,j\};  \, A_{ij} \neq 0\} $ belongs to\footnote{This interpretation of $\bM_\mu$ is justified in the proof of Theorem~\ref{th:MmusplitA}}:
\[
\bM_\mu := \sum_{s=1}^N {\bR\s}^\top \bR\s \bS(\bA) {\bR\s}^\top \bR\s, 
\]
and let $\bB$ be the Hadamard division of $\bA$ by $\bM_\mu$  
\[
B_{ij} := 
\left\{ \begin{array}{l}
A_{ij}/{M_\mu}_{ij} \text{ if } A_{ij} \neq 0 \\
0 \text{ otherwise}
\end{array}\right. 
\text{ for any } i,\,j \in \llbracket 1, n \rrbracket.
\]
Finally, set $\bB\s$ to be the block of $\bB$ corresponding to degrees of freedom in $\Omega\s$:
\[
\bB\s := \bR\s \bB {\bR\s}^\top.
\]
\end{definition} 
Note that from the previous definition only the notation $\bB\s$ will be reused further on in the article. We next check that these matrices $\bB\s$ are indeed suitable.
\begin{theorem}
\label{th:MmusplitA}
Let $\bA$ be an order $n$ spd matrix, let $(\Omega\s)_{s=1,\dots,N}$ represent the partition into subdomains and let $(\bR\s)_{s=1,\dots,N}$ be the set of restriction matrices from Definition~\ref{def:Rs}. Under Assumption~\ref{ass:minoverlap}, the matrices $\bB\s$ from Definition~\ref{def:Bs} satisfy Assumption~\ref{ass:symsplit}.  
\end{theorem}
\begin{proof}
First, we justify the fact that $\bM_\mu$ counts the multiplicity of the pairs of degrees of freedom $\{i,j\}$ for which $A_{ij} \neq 0$:
\[
\left(M_\mu \right)_{ij} = \left(\sum_{s=1}^N {\bR\s}^\top \bR\s \bS(\bA) {\bR\s}^\top \bR\s \right)_{ij} =  \sum\limits_{\{s; \{i,j\}\subset\Omega\s \}}\left( {\bR\s}^\top \bR\s \bS(\bA) {\bR\s}^\top \bR\s \right)_{ij} = \sum\limits_{\{s; \{i,j\}\subset\Omega\s \}} (S(\bA))_{ij} 
\]
so, for any $i,\,j \in \llbracket 1,n \rrbracket$,
\[
\left(M_\mu \right)_{ij} = \left\{ 
\begin{array}{l}
\# \{s; \{i,j\}\subset\Omega\s \} \text{ if } A_{ij}\neq 0, \\
0 \text{ otherwise.}
\end{array}
\right. 
\]
With a similar calculation, it can then be checked that the $\bB\s$ form a splitting of $\bA$:
\[
\left(\sum_{s=1}^N {\bR\s}^\top \bB\s  \bR\s \right)_{ij} =  \left(\sum_{s=1}^N {\bR\s}^\top \bR\s \bB {\bR\s}^\top \bR\s \right)_{ij} =\sum\limits_{\{s; \{i,j\}\subset\Omega\s \}} B_{ij}   
\]
so, for any $i,\,j \in \llbracket 1,n \rrbracket$,
\[
\left(\sum_{s=1}^N {\bR\s}^\top \bB\s  \bR\s \right)_{ij} = \left\{ 
\begin{array}{l}
\sum\limits_{\{s; \{i,j\}\subset\Omega\s \}}A_{ij}/ {M_\mu}_{ij} = A_{ij} \text{ if } A_{ij}\neq 0, \\
0 = A_{ij} \text{ otherwise.}
\end{array}
\right. 
\]
It can be concluded that $\bA = \sum_{s=1}^N {\bR\s}^\top \bB\s  \bR\s $. 
\end{proof}

\begin{remark}
Notice that, if the minimal overlap condition is not satisfied, then there exists a pair of indices $\{i, j\}$ for which $A_{ij} \neq 0$ but $\{s; \{i,j\}\subset\Omega\s \} = \emptyset$. This leads to $\left(\sum_{s=1}^N {\bR\s}^\top \bB\s  \bR\s \right)_{ij} = 0 $ and shows that it is impossible that Assumption~\ref{ass:symsplit} be satisfied without the minimal overlap condition (no matter how the matrices $\bB\s$ are chosen). 
\end{remark}

%
\subsection{Definition of $\bA_+$}
\label{subs:def-of-A+}

The first step in defining the very important matrix $\bA_+$ is to split $\bB\s$ into a positive part and a negative semi-definite part.

\begin{definition}[Splitting of $\bB\s$]
\label{def:As+As-}
Let $\bB\s$, for $\soneN$, be a family of matrices that satisfy Assumption~\ref{ass:symsplit}. For each $s$, let a diagonalization of $\bB\s$ be written as 
\[
\bB\s = \bV\s \bL\s {\bV\s}^\top; \text{ with } \bV\s \text{ orthogonal  and }  \bL\s \text{diagonal}.
\]
Assume, without loss of generality, that the diagonal values of $\bL\s$ (which are the eigenvalues of $\bB\s$) are sorted in non-decreasing order. Let $n\s_+$ be the number of positive eigenvalues and $n\s_- = n\s - n\s_+$ be the number of non-positive eigenvalues. Let $\bV\s_- \in \mathbb R^{n\s \times n\s_-}$, $\bV\s_+\in \mathbb R^{n\s \times n\s_+}$, $\bL\s_-\in \mathbb R^{n\s_- \times n\s_-}$, $\bL\s_+ \in \mathbb R^{n\s_+ \times n\s_+}$ be the blocks of $\bV\s$ and $\bL\s$ that satisfy
\[
\bL\s = \begin{pmatrix} \bL\s_- & \mathbf{0} \\ \mathbf{0} & \bL\s_+ \end{pmatrix}, \quad \bV\s = \left[\bV\s_- | \bV\s_+ \right], \quad \bL\s_+ \text { is spd}, \quad -\bL\s_- \text{ is spsd} . 
\] 
Finally, define the two following matrices in $\mathbb R^{n\s \times n\s}$:
\[
\bA\s_+ := \bV\s_+ \bL\s_+ {\bV\s_+}^\top  \text{ and } \bA\s_- := - \bV\s_- \bL\s_- {\bV\s_-}^\top . 
\]
\end{definition}

It is clear that, for each $\soneN$, both matrices $\bA\s_+$ and $\bA\s_-$ are spsd matrices and that $\bB\s = \bA\s_+ - \bA\s_-$. Next, global matrices are computed by assembling the local components with the usual restriction and prolongation operators $\bR\s$ and ${\bR\s}^\top$.

\begin{definition}[New matrices $\bA_+$ and $\bA_-$]
\label{def:A+A-}
Let the global matrices $\bA_+ \in \mathbb R^{n \times n}$ and $\bA_- \in \mathbb R^{n \times n}$ be defined by  
\[
\bA_+ := \sum_{s=1}^N \bRst \bA\s_+ \bR\s, \text{ and } \bA_- := \sum_{s=1}^N \bRst \bA\s_- \bR\s.  
\]
\end{definition}

\begin{theorem}
The matrices $\bA_+$ and $\bA_-$ from Definition \ref{def:A+A-} satisfy the following three properties
\begin{enumerate}[label=(\roman*)]
\item \label{it:Asplit} $\bA = \bA_+ - \bA_-$ ,
\item \label{it:A-spsd} 
the matrix $\bA_-$ is symmetric positive semi-definite,
\item \label{it:A+spd} 
the matrix $\bA_+$ is symmetric positive definite.
\end{enumerate}
\end{theorem}
\begin{proof}
By definition, $\bA_+ = \sum_{s=1}^N \bRst \bV\s_+ \bL\s_+ {\bV\s_+}^\top   \bR\s$ and  $\bA_- =  - \sum_{s=1}^N \bRst \bV\s_- \bL\s_- {\bV\s_-}^\top   \bR\s$ where the matrices $ \bL\s_+$ ($\soneN$) are diagonal matrices with positive entries and the matrices $\bL\s_-$ ($\soneN$) are diagonal matrices with non-positive entries. Consequently, $\bA_+$ and $\bA_-$ are both spsd and item~\ref{it:A-spsd} is proved. Moreover, by Definition~\ref{def:As+As-} and Assumption \ref{ass:symsplit}, item~\ref{it:Asplit} holds:
\begin{align*}
\bA_+ - \bA_- &= \sum_{s=1}^N \bRst \bV\s_+ \bL\s_+ {\bV\s_+}^\top   \bR\s +   \sum_{s=1}^N \bRst \bV\s_- \bL\s_- {\bV\s_-}^\top   \bR\s \\
&=    \sum_{s=1}^N \bRst \bB\s \bR\s \\
&= \bA. 
\end{align*}

Finally, it has already been argued that $\bA_+$ is spsd so, to prove item~\ref{it:A+spd}, it remains only to confirm that the kernel of $\bA_+$ is restricted to the zero vector. Let $\bx \in \mathbb R^n$ such that $\bA_+ \bx = \mathbf{0}$, 
\[
0 = \langle \bx, \bA_+ \bx \rangle  = \langle \bx, \bA \bx \rangle + \langle \bx, \bA_- \bx \rangle \geq  \langle \bx, \bA\bx \rangle,  
\]
and this last term equals $0$ only if $\bx =\mathbf  0$ which ends the proof.
\end{proof} 

\begin{remark}
The zero eigenvalues of $\bB\s$ are in $\bL\s_-$. Another possibility would be to put them into $\bL\s_+$. 
\end{remark}

The previously defined matrix $\bA_+$ is an spd matrix for which we have knowledge of spsd local matrices $\bN\s = \bA_+\s$ that satisfy \eqref{eq:spsd-splitting} with $C = 1$. This means that it fits right into the abstract GenEO theory \cite{AbstractGenEO} and, hence, a variety of two-level preconditioners with guaranteed convergence rates can be defined. First, the one-level preconditioners to which to apply GenEO are chosen and then the GenEO coarse spaces are given. In terms of the two ingredients still missing in the abstract two-level preconditioners from Section~\ref{subs:abstract-prec}: the local solvers $\tilde \bA\s$ are defined in Section~\ref{subs:def-Honelevel} and the coarse interpolation operators $\bR\0$ are defined in Section~\ref{subs:def-H2}. 

\subsection{One-level preconditioners for $\bA_+$}
\label{subs:def-Honelevel}

In order to define a one-level preconditioner in our framework it remains only to choose the matrices $\tilde \bA\s$ (\textit{i.e.}, the local solvers) in the abstract form \eqref{eq:bH}. Three types of local solvers are introduced as all three are natural choices for $\bA_+$: the exact local solver ($\tilde \bA\s = \bR\s \bA_+ {\bR\s}^\top$), the matrix in the spsd splitting with adequate weights ($\tilde\bA\s = {\bD\s}^{-1} \bA\s_+ {\bD\s}^{-1}$), and what would be the exact solver if we were solving a problem with $\bA$ since this is, after all, our endgame ($\tilde\bA\s = \bR\s \bA {\bR\s}^\top$).

\begin{definition}
\label{def:HASH+ASHNN}
Let three one-level preconditioners be defined by:
\[
\bH_+\AS  := \sum_{s=1}^N {\bR\s}^\top (\bR\s \bA_+ {\bR\s}^\top)^{-1} \bR\s,
\]
\[
\bH\AS  := \sum_{s=1}^N {\bR\s}^\top (\bR\s \bA {\bR\s}^\top)^{-1} \bR\s,
\]
and
\[
\bH\NN  := \sum_{s=1}^N {\bR\s}^\top \bD\s (\bA\s_+)^\dagger \bD\s \bR\s = \sum_{s=1}^N {\bR\s}^\top \bD\s \bV\s_+ (\bL\s_+)\,^{-1} {\bV\s_+}^\top \bD\s \bR\s,
\]
where $\bD\s$ are the partition of unity matrices from Definition~\ref{lem:defofDs}.
\end{definition}

It is recalled that, for $\soneN$, $\bA_+\s$, $\bV\s_+$, and $\bL_+\s$ were introduced in Definition~\ref{def:As+As-}, $\bD\s$ in Definition~\ref{lem:defofDs}, and $\bA_+$ in Definition~\ref{def:A+A-}.

\begin{lemma}
The one-level preconditioners $\bH_+\AS$, $\bH\AS$ and $\bH\NN$ from Definition~\ref{def:HASH+ASHNN} are spd. 
\end{lemma}
\begin{proof}
The preconditioners $\bH_+\AS$ and $\bH\AS$ are usual Additive Schwarz preconditioners for spd matrices so they are spd. For the third preconditioner, it is obvious that $\bH\NN$ is spsd. Moreover, let $\bx \in \operatorname{Ker} (\bH\NN)$ then
\[
0 = \langle \bx , \bH\NN \bx \rangle = \sum_{s=1}^N \langle \bD\s \bR\s  \bx , {\bA\s_+}\,^{\dagger} \bD\s \bR\s  \bx \rangle.
\]
For this to hold, each term in the sum of non-negative terms must also be zero so:
\[
\text{for any } \soneN: \bD\s \bR\s  \bx \in \operatorname{Ker}(\bA\s_+\,^\dagger) =  \operatorname{Ker}(\bA\s_+) . 
\]
Let's prolongate $(\bA\s_+) \bD\s \bR\s \bx = \mathbf 0$ to the global domain with ${\bR\s}^\top$, sum over $s$ and inject the definition of $\bD\s$ (Definition~\ref{lem:defofDs} in which $\bD$ is diagonal) to obtain
\[
\mathbf 0 = \sum_{s=1}^N {\bR\s}^\top \bA\s_+ \bD\s \bR\s \bx =  \sum_{s=1}^N {\bR\s}^\top \bA\s_+  \bR \s \bD {\bR \s}^\top \bR\s \bx =  \sum_{s=1}^N {\bR\s}^\top \bA\s_+  \bR \s \bD \bx =  \bA_+ \bD \bx. 
\]
Finally, the non singularity of $\bA_+$ and of $\bD$ allow to conclude that $\bx = \mathbf 0$ which ends the proof.
\end{proof}

\begin{remark}
For the proof of the non-singularity of $\bH\NN$, the definition of the partition of unity matrices $\bD\s$ was used (Definition~\ref{lem:defofDs}). A general proof does not go through for all partitions of unity (\textit{i.e.,} if the $\bD\s$ in Definition~\ref{def:HASH+ASHNN} are replaced by another family of matrices that satisfy Assumption~\ref{ass:Ds}). However, it is not likely that the $\bD\s$ could and would be chosen in a way that makes $\bH\NN$ singular. In other words, this is a technical restriction and other choices of partition of unity matrices should definitely be explored. All parts of the article that are not related to $\bH\NN$ are not concerned by this technical restriction.  
\end{remark}

\subsection{Two-level preconditioners for $\bA_+$ with GenEO}
\label{subs:def-H2}
Next, the GenEO coarse spaces that correspond to solving a linear system for $\bA_+$ with each of the one-level preconditioners are introduced. The corresponding spectral bounds for the preconditioned operators are given. The proofs consist in giving the adequate references to \cite{AbstractGenEO}. The information is organized with one theorem per choice of one-level preconditioner. First, some very useful notation is chosen to designate a normalized basis of the high (or low) frequency eigenvectors with respect to a certain matrix pencil and a certain threshold. 

\begin{definition}
\label{def:YLYH}
Let $m \in \mathbb N^*$, let $\bM_\bA \in \mathbb R^{m\times m}$ be an spsd matrix, let $\bM_\bB \in \mathbb R^{m\times m}$ be an spd matrix. Let $(\lambda_k, \by_k)_{k=1,\dots,m}$ be the (ordered and $\bM_\bB$-normalized) eigenpairs of the generalized eigenproblem associated with matrix pencil $(\bM_\bA, \bM_\bB)$, \textit{i.e.,}   
\[
\lambda_k \in \mathbb R, \quad \by_k \in \mathbb R^m,\quad \langle \by_k, \bM_B \by_k \rangle = 1,\quad   \lambda_1 \leq \lambda_2 \leq \dots \leq \lambda_m, \quad\text{and} \quad \bM_\bA \by_k = \lambda \bM_\bB \by_k. 
\]

For any scalar $\tau >0$, set $m_L := \min \left\{k\in \llbracket 0, m-1 \rrbracket; \lambda_{k+1} \geq \tau \right\}$ if $\lambda_m \geq \tau$, and $m_L := m$ otherwise. Then define the two following matrices by concatenating eigenvectors 
\[
\bY_L(\tau, \bM_\bA, \bM_\bB) := [\by_1 | \dots |\by_{m_L}], \text{ and } \bY_H(\tau, \bM_\bA, \bM_\bB) := [\by_{m_L + 1} | \dots |\by_m ]. 
\]
\end{definition}
 
The subscripts $*_L$ and $*_H$ refer to the words \textit{low} and \textit{high} depending on which end of the spectrum is selected. The definition is in agreement with the definition in \cite{AbstractGenEO}. 

\begin{remark}
\label{rem:gevpinvMAMB}
It is never necessary to fully solve the generalized eigenvalue problem $\bM_\bA \by_k = \lambda \bM_\bB \by_k$. Instead, only the smallest, or the largest eigenvalues are required as well as the corresponding eigenvectors. This can be performed by an iterative method, many of which are implemented in SLEPc \cite{Hernandez:2005:SSF}. A spectral transformation is performed within these iterative eigensolvers to rewrite the generalized eigenvalue problem in a form that can be solved by a power iteration method \cite{golub13}[Section 7.3]. To this end, 
\begin{itemize}
\item the computation of $\bY_L(\tau, \bM_\bA, \bM_\bB)$ requires to be able to solve linear systems with $\bM_\bA$ and to multiply vectors by $\bM_\bB$,
\item the computation of  $\bY_H(\tau, \bM_\bA, \bM_\bB)$ requires to be able to solve linear systems with $\bM_\bB$ and to multiply vectors by $\bM_\bA$.
\end{itemize}
\end{remark}

\begin{theorem}[Two-level preconditioners with GenEO for $\bH\AS_+$]
For any $\tau_\flat > 1$, let $V\0\uASp(\tau_\flat)$ be defined by
\[
V\0_{\text{AS},+}(\tau_\flat) := \sum_{s=1}^N \range( {\bR\s}^\top \bY_L(\tau_\flat^{-1}, {\bD\s}^{-1} \bA\s_+ {\bD\s}^{-1},    \bR\s \bA_+ {\bR\s}^\top )),
\]
and assume that a corresponding interpolation matrix ${\bR\uASp\0}(\tau_\flat)$ is defined to satisfy Assumption~\ref{ass:V0}. Then, the coarse projector as well as the hybrid and additive preconditioners are defined naturally as 
\begin{eqnarray*}
\bPi\AS_+(\tau_\flat) := &\matid - {\bR\uASp\0}(\tau_\flat)^\top (\bR\uASp\0(\tau_\flat) \bA_+ {\bR\uASp\0}(\tau_\flat)^\top)^{-1} \bR\uASp\0(\tau_\flat) \bA_+, \\
\bH_{+,\mathrm{hyb}}\AS(\tau_\flat)  := &\bPi\AS_+(\tau_\flat)  \bH_+\AS {\bPi_+\AS(\tau_\flat)} ^\top + {\bR\uASp\0(\tau_\flat)}^\top (\bR\uASp\0(\tau_\flat) \bA_+ {\bR\uASp\0(\tau_\flat)}^\top)^{-1} \bR\uASp\0(\tau_\flat), \\
\bH_{+,\mathrm{ad}}\AS(\tau_\flat) := &\bH_+\AS + {\bR\uASp\0(\tau_\flat)}^\top (\bR\uASp\0(\tau_\flat) \bA_+ {\bR\uASp\0(\tau_\flat)}^\top)^{-1} \bR\uASp\0(\tau_\flat).
\end{eqnarray*}

The eigenvalues of the preconditioned operators are bounded as follows
\begin{eqnarray*}
 1/\tau_{\flat} \leq & \lambda(\bH_{+,\mathrm{hyb}}\AS(\tau_\flat) \bA_+ ) &\leq \mathcal N(\bA_+) \\ 
1/((1 + 2 \mathcal N(\bA_+)) \tau_\flat) \leq & \lambda(\bH_{+,\mathrm{ad}}\AS(\tau_\flat) \bA_+ )  &\leq \mathcal N(\bA_+) + 1 ,
\end{eqnarray*}
where $\mathcal N(\bA_+)$ is the coloring constant with respect to matrix $\bA_+$ (see Definition~\ref{def:color}). 
\end{theorem}
\begin{proof}
This results from an application of \cite{AbstractGenEO}[Corollary 3] (for the hybrid preconditioner) and \cite{AbstractGenEO}[Corollary 4] (for the additive preconditioner) under \cite{AbstractGenEO}[Assumption 7]. The parameters are $\mathcal N' = 1$, $\bM\s = {\bD\s}^{-1} \bA\s_+ {\bD\s}^{-1}$, $\tilde \bA \s =  \bR\s \bA_+ {\bR\s}^\top$ and the alternate formulation for the coarse space given in \cite{AbstractGenEO}[Definition 5].
\end{proof}

Note that, if $0 <  \tau_\flat \leq 1$, there is also a spectral result which is slightly longer to state (it involves $\min$ and $\max$).

\begin{theorem}[Two-level preconditioner with GenEO for $\bH\NN$]
\label{th:H2NN}
For any $0 < \tau_\sharp < 1$, let $V\0\uNN(\tau_\sharp)$ be defined by
\[
V\0\uNN(\tau_\sharp) := \sum_{s=1}^N \range( {\bR\s}^\top \bY_L(\tau_\sharp, {\bD\s}^{-1} \bA\s_+ {\bD\s}^{-1},   \bR\s \bA_+ {\bR\s}^\top  )),
\]
and assume that a corresponding interpolation matrix ${\bR\uNN\0}(\tau_\sharp)$ is defined to satisfy Assumption~\ref{ass:V0}. Then, the coarse projector as well as the hybrid preconditioner are defined naturally as 
\begin{eqnarray*}
\bPi\NN(\tau_\sharp) := &\matid - {\bR\uNN\0}(\tau_\sharp)^\top (\bR\uNN\0(\tau_\sharp) \bA_+ {\bR\uNN\0}(\tau_\sharp)^\top)^{-1} \bR\uNN\0(\tau_\sharp) \bA_+, \\
\bH_{\mathrm{hyb}}\NN(\tau_\sharp)  := &\bPi\NN(\tau_\sharp)  \bH\NN {\bPi\NN(\tau_\sharp)} ^\top + {\bR\uNN\0(\tau_\sharp)}^\top (\bR\uNN\0(\tau_\sharp) \bA_+ {\bR\uNN\0(\tau_\sharp)}^\top)^{-1} \bR\uNN\0(\tau_\sharp).
\end{eqnarray*}
The eigenvalues of the preconditioned operator are bounded as follows
\begin{eqnarray*}
 1 \leq & \lambda(\bH_{\mathrm{hyb}}\NN(\tau_\sharp) \bA_+ ) &\leq \mathcal N(\bA_+)/ \tau_\sharp, 
\end{eqnarray*}
where $\mathcal N(\bA_+)$ is the coloring constant with respect to matrix $\bA_+$ (see Definition~\ref{def:color}). 
\end{theorem}
\begin{proof}
This results from an application of \cite{AbstractGenEO}[Corollary 2] with $\mathcal N' = 1$, $\tau_\flat \rightarrow 1$ (as in \cite{AbstractGenEO}[Section 5.2.3]), $\tilde \bA \s = {\bD\s}^{-1}\bA\s_+ {\bD\s}^{-1} $ and $\bR\s \bA_+ {\bR\s}^\top $ as the exact local solver (because we are solving for $\bA_+$). 
\end{proof}

Note that, if $\tau_\sharp \geq 1$, there is also a spectral result which is slightly longer to state (it involves $\min$ and $\max$). There is no spectral result for the two-level additive preconditioner so it is not considered in the theorem.

\begin{theorem}[Two-level preconditioner with GenEO for $\bH\AS$]
For any $\tau_\flat > 1$ and $0 < \tau_\sharp < 1$, let $V\0\uAS(\tau_\sharp,\tau_\flat)$ be defined by
\begin{align*}
V\0\uAS(\tau_\sharp, \tau_\flat) := \sum_{s=1}^N \range&( {\bR\s}^\top \bY_L(\tau_\flat^{-1}, {\bD\s}^{-1} \bA\s_+ {\bD\s}^{-1},  \bR\s \bA {\bR\s}^\top ))\\
&+ \sum_{s=1}^N \range( {\bR\s}^\top \bY_L(\tau_\sharp, \bR\s \bA {\bR\s}^\top,  \bR\s \bA_+ {\bR\s}^\top  )),
\end{align*}
and assume that a corresponding interpolation matrix ${\bR\uAS\0}(\tau_\sharp, \tau_\flat)$ is defined to satisfy Assumption~\ref{ass:V0}. Then, the coarse projector as well as the hybrid preconditioner are defined naturally as 
\begin{eqnarray*}
\bPi\AS(\tau_\sharp,\tau_\flat) := &\matid - {\bR\uAS\0}(\tau_\sharp,\tau_\flat)^\top (\bR\uAS\0(\tau_\sharp,\tau_\flat) \bA_+ {\bR\uAS\0}(\tau_\sharp,\tau_\flat)^\top)^{-1} \bR\uAS\0(\tau_\sharp,\tau_\flat) \bA_+, \\
\bH_{\mathrm{hyb}}\AS(\tau_\sharp,\tau_\flat)  := &\bPi\AS(\tau_\sharp,\tau_\flat)  \bH\AS {\bPi\AS(\tau_\sharp,\tau_\flat)} ^\top + {\bR\uAS\0(\tau_\sharp,\tau_\flat)}^\top (\bR\uAS\0(\tau_\sharp,\tau_\flat) \bA_+ {\bR\uAS\0(\tau_\sharp,\tau_\flat)}^\top)^{-1} \bR\uAS\0(\tau_\sharp,\tau_\flat).
\end{eqnarray*}
The eigenvalues of the preconditioned operator are bounded as follows
\begin{eqnarray*}
1/\tau_\flat  \leq & \lambda(\bH_{\mathrm{hyb}}\AS(\tau_\sharp,\tau_\flat) \bA_+ ) &\leq \mathcal N(\bA_+)/ \tau_\sharp, 
\end{eqnarray*}
where $\mathcal N(\bA_+)$ is the coloring constant with respect to matrix $\bA_+$ (see Definition~\ref{def:color}). 
\end{theorem}
\begin{proof}
This results from \cite{AbstractGenEO}[Corollary 2] with $\mathcal N' = 1$, $\bM\s = {\bD\s}^{-1} \bA\s_+ {\bD\s}^{-1}$, $\tilde \bA \s =  \bR\s \bA {\bR\s}^\top$, $\bR\s \bA_+ {\bR\s}^\top $ as the exact local solver (because we are solving for $\bA_+$), and the alternate formulation for the coarse space given in \cite{AbstractGenEO}[Definition 5].
\end{proof}

Note that, if $0 <  \tau_\flat \leq 1$, there is also a spectral result which is slightly longer to state (it involves $\min$ and $\max$). The choice $\tau_\sharp \geq 1$ would lead to having all the vectors in the coarse space so this is excluded. There is no spectral result for the two-level additive preconditioner so it is not considered in the theorem.

\begin{remark}
For all four considered choices of two-level preconditioner, it is easy to see that $\sum_{s=1}^N {\bR\s}^\top \bD\s \bV\s_-$ is included in the coarse space because $\operatorname{Ker}({\bD\s}^{-1} \bA\s_+ {\bD\s}^{-1}) = \bD\s \bV\s_-$.
\end{remark}
 
\section{Algebraic Woodbury-GenEO (AWG) preconditioners for $\bA$}
\label{sec:newprecs}

In the previous section, four two-level preconditioners for $\bA_+$ (indexed by the choice of one or two thresholds) have been introduced with their spectral bounds: $\bH_{+,\mathrm{hyb}}\AS(\tau_\flat)$, $\bH_{+,\mathrm{ad}}\AS(\tau_\flat)$, $\bH_{\mathrm{hyb}}\NN(\tau_\sharp)$, and $\bH_{\mathrm{hyb}}\AS(\tau_\sharp, \tau_\flat)$. In this section, let $\bHtwo$ denote any one of them and let $[\lambda_{\min}(\bHtwo \bA_+), \lambda_{\max}(\bHtwo \bA_+)]$ be an interval that contains all eigenvalues of the preconditioned operator $\bHtwo \bA_+$. The subscript $*_2$ was chosen to refer to two-level preconditioners. We may now set aside the choice of a preconditioner for $\bA_+$ and come back to our original problem of finding a preconditioner for $\bA$.

\subsection{Woodbury matrix identity for $\bA = \bA_+ - \bA_-$} 

The new preconditioner for $\bA$ arises from the realization that $\bA$ can be viewed as a low rank modification of $\bA_+$ and adding a term to $\bHtwo$ accordingly. 

\begin{theorem}
The rank of $\bA_-$, which we denote by $n_-$, satisfies $n_- \leq \sum_{s=1}^N n\s - n$. 
\end{theorem}
\begin{proof}
By definition of $\bA_+$ and $\bA_-$, it holds that
\[
\operatorname{rank}(\bA_+) + \operatorname{rank}(\bA_-) \leq \sum_{s=1}^N n\s, 
\] 
and $\operatorname{rank}(\bA_+) =n$ as $\bA_+$ is non-singular.
\end{proof}

With words, the rank of $\bA_-$ is at most the difference between the number of degrees of freedom and the number of degrees of freedom multiplied by their multiplicity. If there is little overlap (while still satisfying the minimal overlap condition) then the rank of $\bA_-$ is small compared to the rank $n$ of $\bA$ ($n_- \ll n$). In practice it is possible, and desirable that $n_-$ be much smaller even than $\sum_{s=1}^N n\s - n $. Following this observation it is natural to write $\bA$ as a modification of $\bA_+$. 

To this end, let's introduce the factors in the diagonalization of $\bA_-$.

\begin{definition}
\label{def:V-L-}
Let $\bL_- \in \mathbb R^{n_- \times n_-}$ and $\bV_- \in \mathbb R^{n \times n_-}$ be the diagonal matrix and the orthogonal matrix that are obtained by removing the null part of $\bA_-$ from its diagonalization in such a way that 
\[
\bA_- = \bV_- \bL_- \bV_-^\top. 
\]
\end{definition}

\begin{remark}
The diagonalization of $\bA_-$ is not actually required in the numerical implementation (see Section~\ref{subs:tricks}). 
\end{remark}

It now holds that $\bA = \bA_+ -  \bV_- \bL_- \bV_-^\top$ with $\bA$, $\bA_+$, $\bL_-$ spd matrices and $\bV_-$ a full rank matrix. The Woodbury matrix identity \cite{woodbury1950inverting} applied to computing the inverse of $\bA$ gives
\begin{equation}
\label{eq:Woodbury}
\bA^{-1} = \bA_+^{-1} +  \bA_+^{-1} \bV_- \left(\bL_-^{-1} - \bV_-^\top \bA_+^{-1} \bV_-  \right)^{-1} \bV_-^\top \bA_+^{-1}. 
\end{equation}

\begin{remark}
The Woodbury matrix identity is also called the Sherman-Morrison-Woodbury formula (\textit{e.g.}, in \cite{golub13}[Section 2.1.4]). The formula is correct since $\bA_+$ is non-singular and $\bL_-^{-1} - \bV_-^\top \bA_+^{-1} \bV_- $ is also non-singular. Indeed, let $\by \in \mathbb R^{n_-}$, and assume that $(\bL_-^{-1} - \bV_-^\top \bA_+^{-1} \bV_-)\by = \mathbf 0$,  then
\[
 \bV_- \by = \bV_- \bL_- \bV_-^\top \bA_+^{-1} \bV_- \by \, \Leftrightarrow \,  \bV_- \by = \bA_- \bA_+^{-1} \bV_- \by. 
\]  
With $\bz =  \bA_+^{-1} \bV_- \by$, it then holds that $\bA_+ \bz = \bA_- \bz$ which is equivalent to $\bA \bz = \mathbf 0$ and in turn to $\bz = \mathbf  0$ and $\by = \mathbf  0$.  
\end{remark}

\subsection{AWG preconditioner for $\bA$ with inexact coarse space}

The Woodbury matrix identity leads, rather straightforwardly, to a new preconditioner for the original matrix $\bA$ that is defined in the following theorem.

\begin{theorem}[AWG preconditioner for $\bA$ with inexact coarse space]
Given a preconditioner $\bHtwo$ for $\bA_+$ such that the eigenvalues of $\bHtwo \bA_+$ are in the interval $[\lambda_{\min}(\bHtwo \bA_+), \lambda_{\max}(\bHtwo \bA_+)]$. Let the inexact AWG preconditioner for $\bA$ be defined as 
\[
\bH_{3,\mathrm{inex}} := \bHtwo + \bA_+^{-1} \bV_- \left(\bL_-^{-1} - \bV_-^\top \bA_+^{-1} \bV_-  \right)^{-1} \bV_-^\top \bA_+^{-1}. 
\]
The eigenvalues of the new preconditioned operator satisfy 
\begin{equation}
\label{eq:lambdaWood}
\min (1,\lambda_{\min}(\bHtwo \bA_+)) \leq \lambda(\bH_{3,\mathrm{inex}} \bA ) \leq \max(1, \lambda_{\max}(\bHtwo \bA_+)).
\end{equation}
\end{theorem}
\begin{proof}
The estimate for the eigenvalues of $\bHtwo \bA_+$ is equivalent to 
\[
\lambda_{\min}(\bHtwo \bA_+)   \langle \bx, \bA_+^{-1} \bx \rangle \leq \langle \bx, \bHtwo \bx \rangle \leq \lambda_{\max}(\bHtwo \bA_+) \langle \bx, \bA_+^{-1} \bx \rangle, \,\forall \bx \in \mathbb R^n .
\]
Adding, $ \langle \bx, \bA_+^{-1} \bV_- \left(\bD_-^{-1} - \bV_-^\top \bA_+^{-1} \bV_-  \right)^{-1} \bV_-^\top \bA_+^{-1} \bx \rangle$ to each term, it holds that
\[
\min (1, \lambda_{\min}(\bHtwo \bA_+) )  \langle \bx, \bA^{-1} \bx \rangle \leq \langle \bx, \bH_{3,\mathrm{inex}} \bx \rangle \leq \max (1, \lambda_{\max}(\bHtwo \bA_+)) \langle \bx, \bA^{-1} \bx \rangle, \,\forall \bx \in \mathbb R^n ,
\]
where the Woodbury matrix identity \eqref{eq:Woodbury} was applied. This is equivalent to \eqref{eq:lambdaWood}.
\end{proof}

We have just introduced new preconditioners called the AWG preconditioners with an inexact coarse space. The plural in the previous sentence comes from the fact that there are many possible choices for $\bH_2$ (including the four from Section~\ref{sec:Apos}) and that for each one there are parameters that can be adjusted. These new preconditioners are purely algebraic and they have guaranteed spectral bounds when applied to solving linear system $\bA \bx = \mathbf b$. The condition number of $\bH_{3,\mathrm{inex}} \bA$ can be made smaller by enriching the coarse space in $\bH_2$. The name \textit{inexact} comes from the fact that $\bH_{3,\mathrm{inex}} \bA$ has the form of a domain decomposition preconditioner with two coarse spaces, one is in $\bH_2$ and the other is in the term $\bA_+^{-1} \bV_- \left(\bL_-^{-1} - \bV_-^\top \bA_+^{-1} \bV_-  \right)^{-1} \bV_-^\top \bA_+^{-1}$ where the coarse solve $\left(\bL_-^{-1} - \bV_-^\top \bA_+^{-1} \bV_-  \right)^{-1} $ is inexact. Next, two other AWG preconditioners are defined which have an exact coarse spaces in the sense of the Abstract Schwarz theory: their coarse operator is of the form $\bR\0 \bA {\bR\0}^\top$. 

\subsection{Additive and hybrid AWG preconditioners for $\bA$} 

Solving a problem with the preconditioner introduced in the previous section requires it to be computationally feasible to multiply by $(\bA_+ ^{-1} \bV_-)$ and its transpose. Looking at the Woodbury identity, we realize (and prove it in the lemma below) that $\text{range}\left(\bA_+ ^{-1} \bV_- \right) = \text{range}\left(\bA^{-1} \bV_- \right)$. This opens up new possibilities: if it is possible to compute $\text{range}\left(\bA^{-1} \bV_- \right)$, it is also possible to project $\bA$-orthogonally onto the space that is $\ell_2$-orthogonal to $\range\left(\bV_- \right)$ which is exactly $\text{Ker}(\bA_-)$ (see $\bPi_3$ in Definition~\ref{def:Pi3H3} below). On that space, $\bA_{|\text{Ker}(\bA_-)} = \left(\bA_+\right)_{|\text{Ker}(\bA_-)}$ and we can fall back onto a known and efficient preconditioner $\bH_2$ for $\bA_+$.

\begin{lemma}
\label{lem:rangeinvAV-}
With $\bA_+$ from Definition~\ref{def:A+A-} and $\bV_-$ from Definition~\ref{def:V-L-}, the following property holds
\[
\text{range}\left(\bA_+ ^{-1} \bV_- \right) = \text{range}\left(\bA^{-1} \bV_- \right).
\]
\end{lemma}
\begin{proof}
It follows from the Woodbury identity that  
\[
\text{range}(\bA^{-1}\bV_-) = \text{range}\left(\bA_+^{-1}\bV_-(\matid + \left(\bL_-^{-1} - \bV_-^\top \bA_+^{-1} \bV_-  \right)^{-1} \bV_-^\top \bA_+^{-1}\bV_-) \right), 
\]
where $\matid$ is the $n_- \times n_-$ identity matrix. Proving the result in the lemma comes down to proving that the range of $(\matid + \left(\bL_-^{-1} - \bV_-^\top \bA_+^{-1} \bV_-  \right)^{-1} \bV_-^\top \bA_+^{-1}\bV_-)$ is the whole of $\mathbb R^{n_-}$ or, equivalently, that it's kernel is restricted to $\mathbf 0$ (the zero vector in $\mathbb R^{n_-}$). This last step is achieved as follows. Let $\by \in \mathbb R^{n_-}$,
\begin{align*}
& \left(\matid + \left(\bL_-^{-1} - \bV_-^\top \bA_+^{-1} \bV_-  \right)^{-1} \bV_-^\top \bA_+^{-1}\bV_-\right)\by = \mathbf 0 \\ 
\Leftrightarrow & \left(\bL_-^{-1} - \bV_-^\top \bA_+^{-1} \bV_-  \right)\by =- \bV_-^\top \bA_+^{-1}\bV_- \by \\
 \Leftrightarrow & \bL_-^{-1} \by =  \mathbf 0\\ 
\Leftrightarrow & \by = \mathbf 0.
\end{align*} 
\end{proof}

\begin{definition}
\label{def:Pi3H3}
Let $\bW \in \mathbb R^{n\times n_-}$ be such that $\range (\bW) = \range (\bA_+^{-1} \bV_-)$ and let 
\[
\bPi_3 := \matid - \bW (\bW^\top \bA \bW)^{-1} \bW^\top \bA. 
\]
Assume that $\bHtwo$ is a given preconditioner for $\bA_+$ such that the eigenvalues of $\bHtwo \bA_+$ are in the interval $[\lambda_{\min}(\bHtwo \bA_+), \lambda_{\max}(\bHtwo \bA_+)]$. Let two new preconditioner for $\bA$ be defined as 
\[
\bH_{3,\mathrm{ad}} :=  \bHtwo + \bW (\bW^\top \bA \bW)^{-1} \bW^\top \text{ (Additive AWG preconditioner)}, 
\]
and
\[
\bH_{3,\mathrm{hyb}} := \bPi_3 \bHtwo \bPi_3^\top + \bW (\bW^\top \bA \bW)^{-1} \bW^\top \text{ (Hybrid AWG preconditioner)}. 
\]
\end{definition}

\begin{theorem}
Let $\bPi_3$, $\bH_{3,\mathrm{ad}}$ and $\bH_{3,\mathrm{hyb}}$ be as in Definition~\ref{def:Pi3H3}. The operator $\bPi_3$ is an $\bA$-orthogonal projection operator that satisfies
\begin{equation}
\text{Ker}(\bPi_3) = \text{range}\left(\bA^{-1}\bV_-\right) \text{ and } \text{range}(\bPi_3) = \text{Ker}\left(\bA_-\right).
\label{eq:prop-Pi3}
\end{equation}
Moreover, the new preconditioned operators satisfy the spectral bounds :
\begin{eqnarray}
\lambda_{\min}(\bHtwo \bA_+) \leq &\lambda( \bHtwo \bA \bPi_3 ) &\leq \lambda_{\max}(\bHtwo \bA_+) \text{ if } \lambda(\bHtwo \bA \bPi_3) \neq 0, \label{eq:H2APi3}\\ 
\min(1, \lambda_{\min}(\bHtwo \bA_+)  ) \leq & \lambda(\bH_{3,\mathrm{hyb}} \bA   ) &\leq \max(1, \lambda_{\max}(\bHtwo \bA_+)  ),  \label{eq:H3hybA} \\ 
 \min (1, \lambda_{\min}(\bHtwo \bA_+))  \leq &\lambda( \bH_{3,\mathrm{ad}} \bA    ) &\leq  (\lambda_{\max}(\bHtwo \bA_+) + 1) \label{eq:H3adA} .
\end{eqnarray}
where we recall that $\bHtwo$ can be chosen as one of the two-level preconditioners from Section~\ref{subs:def-H2} in such a way that $\lambda_{\min}(\bHtwo \bA_+)$ and  $\lambda_{\max}(\bHtwo \bA_+)$ are known and controlled by the choice of the coarse space. 
\end{theorem}
Note that a bound for the projected and preconditioned operator $( \bHtwo \bA \bPi_3 )$ has also been included in the theorem (equation \eqref{eq:H2APi3}). 
\begin{proof}
We begin by proving \eqref{eq:prop-Pi3}. Let $\bx \in \mathbb R^n$, $\bx$ is in the kernel of $\bPi_3$ if:
\[
\bPi_3 \bx = 0 \Leftrightarrow \bx= \bW (\bW^\top \bA \bW)^{-1} \bW^\top \bA \bx \Leftrightarrow \bx \in \text{range}(\bW) = \text{range}(\bA_+^{-1}\bV_-) = \text{range}(\bA^{-1}\bV_-).
\]
The last equality comes from Lemma~\ref{lem:rangeinvAV-}. Moreover, $\bPi_3$ is an $\bA$-orthogonal projection so 
\[
\text{range}(\bPi_3) = \left( \text{Ker}(\bPi_3) \right)^{\perp^\bA} = \left( \text{range}(\bA^{-1}\bV_-) \right)^{\perp^\bA} = \left( \text{range}(\bV_-) \right)^{\perp^{\ell_2}}= \text{Ker}\left(\bV_-^\top \right) =  \text{Ker}\left(\bA_-\right), 
\] 
by definition of $\bV_-$ in Definition~\ref{def:V-L-}. A direct consequence of this result, that is frequently used in the remainder of the proof, is the identity $\bA_+ \bPi_3 = \bA \bPi_3$. 

We now move onto proving the spectral bounds starting with \eqref{eq:H2APi3} for the projected and preconditioned operator $\bHtwo \bA \bPi_3$. Let $(\lambda, \by) \in \mathbb R \times \mathbb R^n$ be an eigenpair of the matrix $\bHtwo \bA \bPi_3 (= \bHtwo \bA_+ \bPi_3)$ meaning that:
\[
 \by \neq \mathbf 0 \text{ and } \bHtwo \bA_+ \bPi_3 \by = \lambda \by.
\]
Taking the inner product by $ \bA_+ \bPi_3 \by $ gives
\[
\langle \bA_+ \bPi_3 \by , \bHtwo \bA_+ \bPi_3 \by \rangle = \lambda \langle \bA_+ \bPi_3 \by, \by \rangle = \lambda \langle \bA \bPi_3 \by, \by \rangle =  \lambda \langle \bA \bPi_3 \by, \bPi_3 \by \rangle   =  \lambda \langle \bA_+ \bPi_3 \by, \bPi_3 \by \rangle.   
\]
Moreover, since $\bA_+$ and $\bHtwo$ are spd, the spectral bound for $\bHtwo \bA_+$ is equivalent to
\begin{equation}
\label{eq:intermA+H2}
\lambda_{\min} (\bHtwo \bA_+) \langle \bx, \bA_+ \bx \rangle \leq  \langle \bx, \bA_+ \bHtwo \bA_+ \bx \rangle \leq  \lambda_{\max} (\bHtwo \bA_+) \langle \bx, \bA_+ \bx \rangle, \quad \forall \, \bx \in \mathbb R^n. 
\end{equation}
In particular, for $\bx = \bPi_3\by$ we get 
\[
\lambda_{\min} (\bHtwo \bA_+) \langle \bPi_3 \by, \bA_+ \bPi_3 \by \rangle \leq  \underbrace{\langle\bPi_3 \by, \bA_+ \bHtwo \bA_+ \bPi_3 \by \rangle}_{=\lambda \langle \bA_+ \bPi_3 \by, \bPi_3 \by \rangle  } \leq  \lambda_{\max} (\bHtwo \bA_+) \langle \bPi_3 \by, \bA_+ \bPi_3 \by \rangle. 
\]
Finally, there are two possibilities, either $\langle \bA_+ \bPi_3 \by, \bPi_3 \by \rangle = 0$, so $\by \in \operatorname{Ker}(\bPi_3)$ and $\lambda = 0$, or $\langle \bA_+ \bPi_3 \by, \bPi_3 \by \rangle \neq 0$ and $\lambda \in [\lambda_{\min}(\bHtwo \bA_+), \lambda_{\max}(\bHtwo \bA_+)]$. In other words \eqref{eq:H2APi3} holds.

Next, we prove the spectral bound with the hybrid preconditioner $\bH_{3,\mathrm{hyb}}$ from \eqref{eq:H3hybA}. Let $\bx \in \mathbb R^n$, we add the term 
\[
\langle \bx, \bA \bW (\bW^\top \bA \bW)^{-1} \bW^\top \bA \bx \rangle = \langle \bx, \bA (\matid - \bPi_3) \bx \rangle   = \langle (\matid - \bPi_3) \bx, \bA (\matid - \bPi_3) \bx \rangle   
\]
to estimate \eqref{eq:intermA+H2} evaluated at $\bPi_3 \bx$ and obtain
\begin{align*}
\lambda_{\min} (\bHtwo \bA_+) \langle \bPi_3 \bx, \bA \bPi_3 \bx \rangle + &\langle (\matid - \bPi_3) \bx, \bA (\matid - \bPi_3) \bx \rangle    \leq \\
 \langle \bPi_3 \bx, \bA \bHtwo &\bA \bPi_3 \bx \rangle + \langle \bx, \bA \bW (\bW^\top \bA \bW)^{-1} \bW^\top \bA \bx \rangle   \leq \\
& \lambda_{\max} (\bHtwo \bA_+) \langle \bPi_3 \bx, \bA \bPi_3 \bx \rangle + \langle (\matid - \bPi_3) \bx, \bA (\matid - \bPi_3) \bx \rangle  
\end{align*}
where $\bA_+ \bPi_3 = \bA \bPi_3$ as also been applied. This then implies that 
\[
\min(1, \lambda_{\min}(\bHtwo \bA_+)  ) \langle \bx, \bA \bx \rangle \leq  \langle\bx, \bA \bH_{3,\mathrm{hyb}} \bA\bx \rangle \leq  \max(1, \lambda_{\max}(\bHtwo \bA_+)  ) \langle \bx, \bA \bx \rangle . 
\]
and the eigenvalue estimate in the theorem holds because $\bA$ and $\bH_{3,\mathrm{hyb}}$ are spd.

Finally, we prove the spectral bound with the additive preconditioner $\bH_{3,\mathrm{ad}}$ from \eqref{eq:H3hybA}. Matrices $\bA_+$ and $\bHtwo$ are both spd so the fact that all eigenvalues are not greater than $\lambda_{\max}(\bHtwo \bA_+)$ is equivalent to
\[
\langle \bx, \bHtwo \bx \rangle \leq \lambda_{\max}(\bHtwo \bA_+) \langle \bx, {\bA_+}^{-1} \bx \rangle, \text{ for any } \bx \in \mathbb R^n.
\]
Moreover, $\langle \bx, {\bA_+}^{-1} \bx \rangle \leq \langle \bx, {\bA}^{-1} \bx \rangle  \text{ for any } \bx \in \mathbb R^n$ so 
\[
\langle \bx, \bHtwo \bx \rangle \leq \lambda_{\max}(\bHtwo \bA_+) \langle \bx, {\bA}^{-1} \bx \rangle, \text{ for any } \bx \in \mathbb R^n.
\]
It also holds that
\[
\langle \bx, \bA \bW (\bW^\top \bA \bW)^{-1} \bW^\top \bA \bx \rangle = \langle (\matid - \bPi_3) \bx, \bA (\matid - \bPi_3) \bx \rangle  \leq \langle \bx, \bA \bx \rangle , \text{ for any } \bx \in \mathbb R^n, 
\] 
or equivalently, 
\[
\langle \bx, \bW (\bW^\top \bA \bW)^{-1} \bW^\top \bx \rangle  \leq \langle \bx, \bA^{-1} \bx \rangle , \text{ for any } \bx \in \mathbb R^n. 
\] 
Adding the last two results together gives us
\[
\langle \bx, \bH_{3,\mathrm{ad}} \bx \rangle \leq (\lambda_{\max}(\bHtwo \bA_+) + 1)\langle \bx, {\bA}^{-1} \bx \rangle, \text{ for any } \bx \in \mathbb R^n,
\]
or in other words, all eigenvalues of $\bH_{3,\mathrm{ad}} \bA$ are not greater than $ (\lambda_{\max}(\bHtwo \bA_+) + 1)$.
For the smallest eigenvalue of $\bH_{3,\mathrm{ad}} \bA$, we can look at $\bH_{3,\mathrm{ad}}$ in the abstract Schwarz framework. Indeed,
\[
\bH_{3,\mathrm{ad}} = \matid \bHtwo \matid + \bW (\bW^\top \bA \bW)^{-1} \bW^\top, 
\]
is an abstract Schwarz solver for $2$ subspaces $\mathbb R^n$ and $\mathbb R^{n_-}$ with prolongation operators $\matid$ (identity matrix in $\mathbb R^n$) and $\bW$, and with local solvers $\bHtwo$ and $ (\bW^\top \bA \bW)^{-1} $. We know that $\bHtwo$ is spd so the classical stable splitting result from Toselli Widlund \cite{ToselliWidlund_book2005} applies (or from \cite{AbstractGenEO}): all eigenvalues of $\bH_{3,\mathrm{ad}} \bA$ are larger than $C_0^{-2}$ if, for any $\bx \in \mathbb R^n$, there exist $\bz_+ \in \mathbb R^n$ and $\bz_- \in \mathbb R^{n_-}$ that satisfy
\[
\bz_+ + \bW \bz_- = \bx \text{ and }\langle \bz_+, \bHtwo^{-1} \bz_+ \rangle + \langle \bz_-, \bW^\top \bA \bW \bz_- \rangle \leq C_0^2 \langle \bx, \bA \bx \rangle. 
\] 
The following splitting is proposed: $\bz_+ =  \bPi_3 \bx$ and $\bz_- = (\bW^\top \bA \bW)^{-1} \bW^\top \bA \bx$. We first check that they do split $\bx$ :
\[
\bz_+ + \bW \bz_- = (\matid - \bPi_3) \bx + \bW (\bW^\top \bA \bW)^{-1} \bW^\top \bA \bx = \bx. 
\]
We then check the stability of the splitting:
\begin{align*}
\langle \bz_+, \bHtwo^{-1} \bz_+ \rangle + \langle \bz_-, \bW^\top \bA \bW \bz_- \rangle  & = \langle \bPi_3 \bx , \bHtwo^{-1} \bPi_3 \bx \rangle + \langle (\bW^\top \bA \bW)^{-1} \bW^\top \bA \bx, \bW^\top \bA \bW (\bW^\top \bA \bW)^{-1} \bW^\top \bA \bx \rangle \\
& \leq (\lambda_{\min}(\bHtwo \bA_+))^{-1}  \langle \bPi_3 \bx , \bA_+ \bPi_3 \bx \rangle + \langle (\matid - \bPi_3) \bx , \bA (\matid - \bPi_3) \bx \rangle\\ 
& \leq (\lambda_{\min}(\bHtwo \bA_+))^{-1}  \langle \bPi_3 \bx , \bA \bPi_3 \bx \rangle + \langle (\matid - \bPi_3) \bx , \bA (\matid - \bPi_3) \bx \rangle \\
& \leq  \max (1, \lambda_{\min}(\bHtwo \bA_+)^{-1})  \langle  \bx , \bA \bx \rangle .
\end{align*}
Finally, we have proved that all eigenvalues of $\bH_{3,\mathrm{ad}}$ are greater than or equal to $ \min (1, \lambda_{\min}(\bHtwo \bA_+))$.
\end{proof}

Again, the AWG preconditioners $\bH_{3,\mathrm{ad}}$ and $\bH_{3,\mathrm{hyb}}$ are families of preconditioners that are computed algebraically and lead to guaranteed spectral bounds when applied to solving $\bA \bx = \mathbf b$. The condition numbers can be made smaller by enriching the coarse space in $\bH_2$.

\subsection{Remarks on the implementation of the AWG preconditioners}
\label{subs:tricks}

Below, some important remarks are made about the implementation of $\bH_{3,\mathrm{ad}}$ and $\bH_{3,\mathrm{hyb}}$. 

\begin{enumerate}
\item It is not necessary to diagonalize $\bA_-$ as suggested by the definition of $\bV_-$ in Definition~\ref{def:V-L-}. Indeed, Definition~\ref{def:Pi3H3} requires only a basis $\bW$ of $ \range (\bA_+^{-1} \bV_-)$ to generate the second coarse space. A natural choice is to recall that $\bV_-$ is generated by the eigenvectors of $\bB\s$ that correspond to negative eigenvalues (once prolongated to $\Omega$ by ${\bR\s}^\top$). If the matrices $\bB\s$ are non-singular then the range of $ (\bA_+^{-1} \bV_-)$ is also generated by 
\[
\bW = \bA_+^{-1} [{\bR^1}^\top \bV^1_- \,|\, \dots \,|\, {\bR^N}^\top \bV^N_- ], 
\]
with the $\bV^s_-$ from Definition~\ref{def:As+As-}. If the $\bB\s$ are singular, it is necessary to first remove the columns in $\bV\s_-$ that correspond to zero eigenvalues. It may also be necessary (although we haven't observed it in practice) to remove some linear dependencies between the columns. This is rather standard and can be done either when computing $\bW$ or when factorizing the coarse problem $\bW^\top \bA \bW$.
\item The computation of $\bW$ is one of the bottlenecks of the algorithm: many systems must be solved for the global matrix $\bA_+$. In our current implementation these linear systems are solved one after the other with PCG preconditioned by $\bH_2$. Since $\bH_2$ is a \textit{good} preconditioner for $\bA_+$ this takes \textit{few} iterations. It must be explored whether computational efficiency could be improved with block CG methods \cite{o1980block} or adaptive multipreconditioning \cite{spillane2016adaptive}.  
\item Following Remark~\ref{rem:gevpinvMAMB}, all four choices of preconditioners $\bH_2$ for $\bA_+$ that are considered in the article require that the action of $(\bA\s_+)^\dagger$ be implemented in order to compute the corresponding GenEO coarse space. Instead of computing the full diagonalization of $\bB\s$, it is sufficient to compute its negative eigenvalues and corresponding orthonormalized set of eigenvectors (\textit{i.e.}, $\bL\s_-$ and $\bV\s_-$ from Definition~\ref{def:As+As-}) and to recall that 
\[
\bA\s_+ = (\matid - \bV\s_- {\bV\s_-}^\top) \bB\s (\matid - \bV\s_- {\bV\s_-}^\top)
\] 
which also implies that 
\[
{\bA\s_+}^\dagger = (\matid - \bV\s_- {\bV\s_-}^\top) {\bB\s}^\dagger (\matid - \bV\s_- {\bV\s_-}^\top).
\]
Since $\bB\s$ is symmetric, it can be factorized using MUMPS \cite{MUMPS:1,MUMPS:2}.
\item Following Remark~\ref{rem:gevpinvMAMB}, all four choices of preconditioners $\bH_2$ for $\bA_+$ that are considered in the article require that the action of $\bR\s \bA_+ {\bR\s}^\top$ be implemented in order to compute the corresponding GenEO coarse space.  As $\bR\s \bA_+ {\bR\s}^\top$ is a dense matrix, it is never assembled. Instead, the action of $\bR\s \bA_+ {\bR\s}^\top$ is computed as
\[
\bR\s \bA_+ {\bR\s}^\top = \bR\s \bA {\bR\s}^\top + \bR\s \bA_- {\bR\s}^\top = \bR\s \bA {\bR\s}^\top - \sum_{t=1}^N  \bR\s {\bR^t}^\top  \bV^t_- \bL^t_- {\bV^t_-}^\top   \bR^t   {\bR\s}^\top,
\] 
where again $\bL^t_-$ and $\bV^t_-$ are the ones from Definition~\ref{def:As+As-}.
In the sum, all terms for which $\bR^t   {\bR\s}^\top$ is zero are zero. 
\item If $ \bH_+\AS  := \sum_{s=1}^N {\bR\s}^\top (\bR\s \bA_+ {\bR\s}^\top)^{-1} \bR\s$ is chosen as a one-level preconditioner for $\bA_+$ then it is necessary to compute the action of $(\bR\s \bA_+ {\bR\s}^\top)^{-1}$. This is done by applying the Woodbury matrix identity to the formula just above. 

\end{enumerate}

\section{Numerical Results}
\label{sec:Numerical}

In this section, numerical results are presented for the new AWG preconditioners with exact coarse spaces: $\bH_{3,\mathrm{ad}}$ and $\bH_{3,\mathrm{hyb}}$. The theoretical convergence bounds are checked, and the behaviour of the new preconditioners is illustrated for the first time. Some comparisons to non-algebraic domain decomposition preconditioners with more standard GenEO coarse spaces are performed. The linear systems that are considered result from discretizing a two-dimensional linear elasticity problem with $Q_1$ finite elements. All details are given below. 

\begin{remark}
The AWG preconditioner with inexact coarse space has not been included into the numerical study but some numerical results can be found in \cite{spillane:hal-03187092}. The behaviour of $\bH_{3,\mathrm{inex}}$ is not expected to differ much from the behaviour of  $\bH_{3,\mathrm{ad}}$ and $\bH_{3,\mathrm{hyb}}$. In particular they all share the same coarse spaces and have very similar convergence bounds (or exactly the same in the case of $\bH_{3,\mathrm{hyb}}$). In the future, when CPU time is considered, $\bH_{3,\mathrm{inex}}$ should be included in the comparison. 
\end{remark}

All the results presented were obtained with petsc4py \cite{dalcin2011parallel}, a Python port to the PETSc libraries \cite{petsc-web-page,petsc-user-ref,petsc-efficient}.  The eigensolves are performed by SLEPc \cite{Hernandez:2005:SSF} and the matrix factorizations (for the local and coarse problems) are performed by MUMPS \cite{MUMPS:1,MUMPS:2}. Our code is available on Github \cite{ourcode}.

Let $\omega = [0 , 3] \times [0,3]  \subset \mathbb R^2$  be the computational domain. Let $\partial \omega_D$ be the left hand side boundary of $\omega$ and let $\mathcal V =  \{\bv \in H^1(\omega)^2; \bv= \mathbf 0 \text{ on } \partial\omega_D\}$.  A solution $\bu  \in \mathcal V$ is sought such that
\begin{equation}
\label{eq:elasticity}
\int_{\omega} 2 \mu \varepsilon(\bu) : \varepsilon(\bv) \, dx +  \int_{\omega} L \operatorname{div}(\bu) \operatorname{div}(\bv) \, dx  = \int_{\omega} \mathbf{g} \cdot \bv \, dx, \text{ for all } \bv \in \mathcal V,
\end{equation}
where, for $i,j=1,2$, $\varepsilon_{ij}(\bu) = \frac{1}{2}\left(\frac{\partial {u}_i}{\partial x_j}+\frac{\partial {u}_j}{\partial x_i}\right)$, $\delta_{ij}$ is the Kronecker symbol, $\mathbf g = (0,-9.81)^\top$ and the Lam\'e coefficients are functions of Young's modulus $E$ and Poisson's ratio $\nu$ : $  
\mu = \frac{E}{2(1+\nu)},\, L = \frac{E\nu}{(1+\nu)(1-2\nu)}$. It is well known (see, \textit{e.g.}, \cite{pechstein2011analysis}) that the solution of \eqref{eq:elasticity} in a heterogeneous medium is challenging due to ill-conditioning. Unless otherwise specified, the coefficient distribution that is considered is the following: for any $(x,y) \in \omega$, 
\begin{equation}
\label{eq:distribE}
\nu(x,y) = 0.3 \quad  \text{ and } E(x,y) = \left\{\begin{array}{ll} 10^{11}& \text{ if } (\operatorname{floor}(y) - y) \in [1/7,2/7]\cup  [3/7,4/7],  \\ 10^7& \text{otherwise.}\end{array}   \right. 
\end{equation}

\begin{figure}
\begin{center}
\includegraphics[width=0.7\textwidth, trim = 4cm 4cm 4cm 4cm, clip]{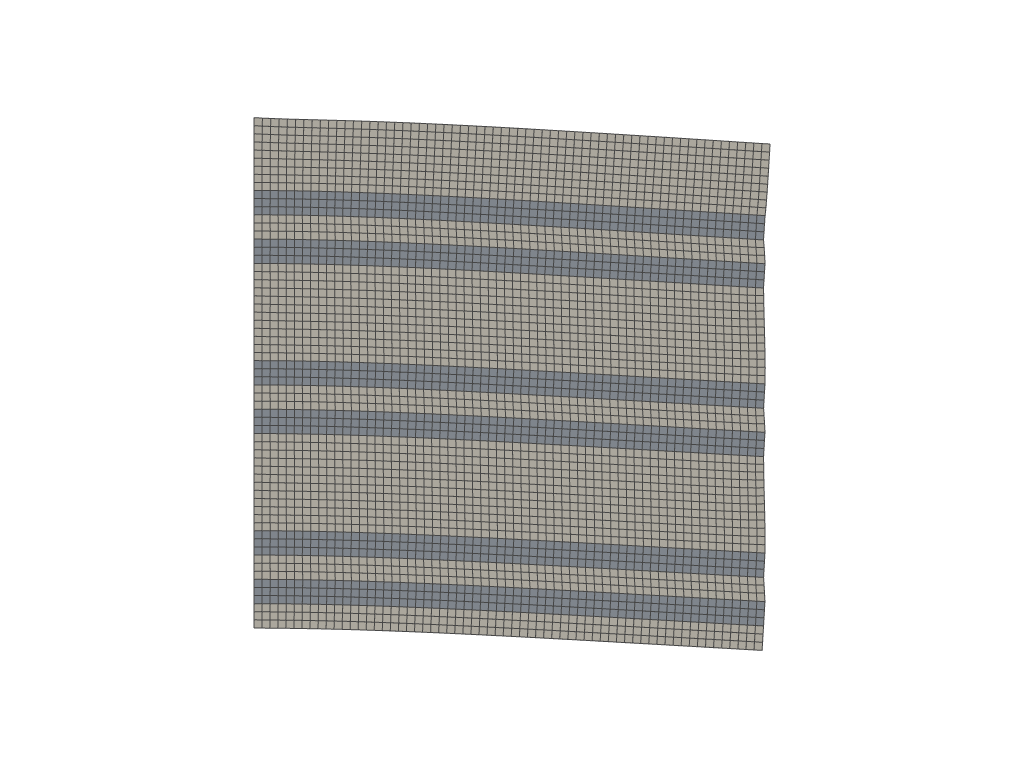} 
\end{center}
\caption{The computational domain $\omega$ has been displaced by $10^5$ multiplied by the solution $\bx_*$. The colors show the distribution of coefficients as defined in \eqref{eq:distribE}. Dark: $E = E_1 = 10^{11}$ - Light :  $E = E_2 = 10^{7}$. }
\label{fig:E-solution}
\end{figure}

The computational domain is discretized by a uniform mesh with element size $h = 1/21$ so there are  
\[
n = 8064 \text{ degrees of freedom (once removed the ones on $\partial \omega_D $)}. 
\] 
The boundary value problem is solved numerically with $Q_1$ finite elements. Let $\mathcal V_h$ be the space of $Q_1$ finite elements that satisfy the Dirichlet boundary condition. Let $\{\bphi_k\}_{k=1}^n$ be a basis of $\mathcal V_h$. The linear system that is to be solved is
\[
\text{Find } \bx_* \in \mathbb R^n \text{ such that } \bA \bx_* = \bb,
\] 
with $A_{ij} = \int_{\omega} \left[ 2 \mu \varepsilon(\bphi_i) : \varepsilon(\bphi_j) +  L \operatorname{div}(\bphi_i) \operatorname{div}(\bphi_j) \right] \, dx $ and $\bb_i = \int_{\omega} \mathbf{g} \cdot \bphi_i \, dx$. The mesh, the solution and the distribution of $E$ are represented in Figure~\ref{fig:E-solution}.

Unless otherwise specified, for each computation, the domain $\omega$ is split into 
\[
N = 9
\]
unit square subdomains that overlap only at the interface and this in turn gives the partition of the degrees of freedom into $\Omega^1$, \dots, $\Omega^9$. With $Q_1$ finite elements, for the overlap to be restricted to the shared subdomain boundaries is enough to ensure the minimal overlap condition. There are $504$ degrees of freedom that are shared by more than one subdomain. All linear systems, are solved with PCG up to a relative residual tolerance of $10^{-10}$. The preconditioner is specified for each test case. 

The matrix $\bA_+^{-1} \bV_-$ is computed by solving $n_-$ linear systems for $\bA_+^{-1}$. This is also done with PCG preconditioned by $\bH_2$ up to a relative residual tolerance of $10^{-10}$ (unless specified otherwise).

Except in the next paragraph, the AWG preconditioner under study is 
\[
\bH_{3,\mathrm{ad}} \text{ with } \bH_2 = \bH_{\mathrm{hyb}}\NN(\tau_\sharp) \text{ (from Definition~\ref{def:Pi3H3} and Theorem~\ref{th:H2NN})}.
\]

\paragraph{Comparison of $\bH_{3,\mathrm{ad}}$ and $\bH_{3,\mathrm{hyb}}$ for all variants of $\bH_2$}

The test case is solved with the eight AWG preconditioners (for fixed values of threshold). Specifically, there are two ways of incorporating the second coarse space leading to $\bH_{3,\mathrm{ad}}$ (additive) and $\bH_{3,\mathrm{hyb}}$ (hybrid) as well as, for each one, four choices for $\bH_2$: $\bH_{\mathrm{hyb}}\NN(\tau_\sharp) $, $\bH_{\mathrm{hyb}}\AS(\tau_\sharp,\tau_\flat)$, $\bH_{+,\mathrm{ad}}\AS(\tau_\flat)$ and $\bH_{+,\mathrm{hyb}}\AS(\tau_\flat)$. The thresholds for selecting eigenvalues in the GenEO coarse spaces are set to $\tau_\flat = 10$ and $\tau_\sharp = 0.1$. The results are shown in Table~\ref{tab:allprecs}. As a matter of comparison, results with more classical (non-algebraic) domain decomposition preconditioners with GenEO coarse spaces presented in \cite{AbstractGenEO}[Section 5] are also reported. 

In all lines of the table that correspond to AWG, the quantity $n_-$ (size of the second coarse space) is the same which is normal because the second coarse space depends only on $\bA$. The size $\#V_0$ of the GenEO coarse space also appears to be the same for all choices of $\bH_2$. For $\bH_{\mathrm{hyb}}\NN(\tau_\sharp) $, $\bH_{\mathrm{hyb}}\AS(\tau_\sharp,\tau_\flat)$, $\bH_{+,\mathrm{ad}}\AS(\tau_\flat)$, the same eigenvalue problem is being solved so as long as $\tau_\sharp = \tau_\flat^{-1}$ this was expected. For the last choice $\bH_2= \bH_{\mathrm{hyb}}\AS(\tau_\sharp,\tau_\flat)$, it is not entirely surprising that the size of the coarse space is not too different as there are connections between the GenEO eigenproblems but a small difference in size would not have surprised us either. 

For all AWG preconditioners, the extreme eigenvalues of the preconditioned operators behave as predicted. All AWG preconditioners reduce the condition number of the preconditioned operator to a very small value below $20$ with the result that convergence to $10^{-10}$ occurs in at most $31$ iterations. The GenEO coarse spaces constructed by AWG are of almost the same size as the classical (non algebraic) coarse spaces which is very satisfying. Of course the AWG preconditioners bear the cost of the extra coarse space.  

For $\bH_2 = \bH_{\mathrm{hyb}}\NN(\tau_\sharp)$ the first 20 non-zero eigenvalues computed for the GenEO eigenproblem in each subdomain are plotted in Figure~\ref{fig:PCNewNN-eigenvalues}. It appears that choosing larger values of $\tau_\sharp$ than $10^{-1}$ could increase significantly the size of the coarse space without much improving the condition number much as the eigenvalues are quite clustered. The eigenvectors that are selected for the coarse space are plotted in Figure~\ref{fig:PCNewNN-CV4}. Since the unknowns are displacements, they have been represented by applying the deformation to the subdomain. The colors show the values of $E$. The influence of the hard (darker colored) layers can be seen but it is not easy to make any conclusions about the eigenvectors. 

\begin{table}
\begin{center}
\begin{tabular}{lrrrrrr}
\hline            &   $\kappa$ &  $It$ &  $\lambda_{\min}$ &  $\lambda_{\max}$ & $\#V_0$ &  $n_- $\\
\hline
\multicolumn{7}{l}{New AWG preconditioners:}\\
\hline
  $\bH_{3,\mathrm{ad}}$ with $\bH_2 = \bH_{\mathrm{hyb}}\NN(0.1)$ &   9.09 &    26 &   1.0 &   9.1 &     57 &                 48 \\
  $\bH_{3,\mathrm{ad}}$ with $\bH_2 = \bH_{\mathrm{hyb}}\AS(0.1,10)$&  12.2 &    26 &   0.33 &   4.0 &     57 &                 48 \\
  $\bH_{3,\mathrm{ad}}$ with $\bH_2 = \bH_{+,\mathrm{hyb}}\AS(10)$  &  12.3 &    25 &   0.33 &   4.0 &     57 &                 48 \\
  $\bH_{3,\mathrm{ad}}$ with $\bH_2 = \bH_{+,\mathrm{ad}}\AS(10) $  &  16.8 &    31 &   0.24 &   4.0 &     57 &                 48 \\
  $\bH_{3,\mathrm{hyb}}$ with $\bH_2 = \bH_{\mathrm{hyb}}\NN(0.1)$   &   9.09 &    27 &   1.0 &   9.1 &     57 &                 48 \\
  $\bH_{3,\mathrm{hyb}}$ with $\bH_2 = \bH_{\mathrm{hyb}}\AS(0.1,10)$ &  12.1 &    25 &   0.33 &   4.0 &     57 &                 48 \\
  $\bH_{3,\mathrm{hyb}}$ with $\bH_2 = \bH_{+,\mathrm{hyb}}\AS(10)$   &  12.2 &    25 &   0.33 &   4.0 &     57 &                 48 \\
  $\bH_{3,\mathrm{hyb}}$ with $\bH_2 = \bH_{+,\mathrm{ad}}\AS(10) $   &  16.7 &    29 &   0.24 &   4.0 &     57 &                 48 \\
\hline
\multicolumn{7}{l}{Non-algebraic methods:}\\
\hline
      Hybrid AS + GenEO ($\tau =10$) &  26.5 &    43 &   0.15 &   4.0 &     55 &                  0 \\
     Additive AS + GenEO ($\tau =10$) &  50.0 &    58 &   0.080 &   4.0 &     55 &                  0 \\
       BNN with GenEO ($\tau =0.1$)&  11.1 &    29 &   1.0 &  11.1 &     55 &                  0 \\
  One-level AS &  34772 &  $>$ 150 &   0.000115 &   4.0 &      0 &                  0 \\

\hline
\end{tabular}
\end{center}
\caption{\label{tab:allprecs}Comparison between all AWG preconditioners for fixed $\tau_\sharp$ and $\tau_\flat$. Data is also included for classical GenEO and the one-level method.   $\kappa$: condition number of preconditioned operator, $It$: number of iterations, $\lambda_{\min}$: smallest eigenvalue of preconditioned operator,  $\lambda_{\max}$: largest eigenvalue of preconditioned operator, $\#V_0$: dimension of GenEO coarse space,  $n_- = \operatorname{rank}(\bA_-)$: dimension of second coarse space.}
\end{table}

\begin{figure}
\begin{center}
\includegraphics[width=0.7\textwidth]{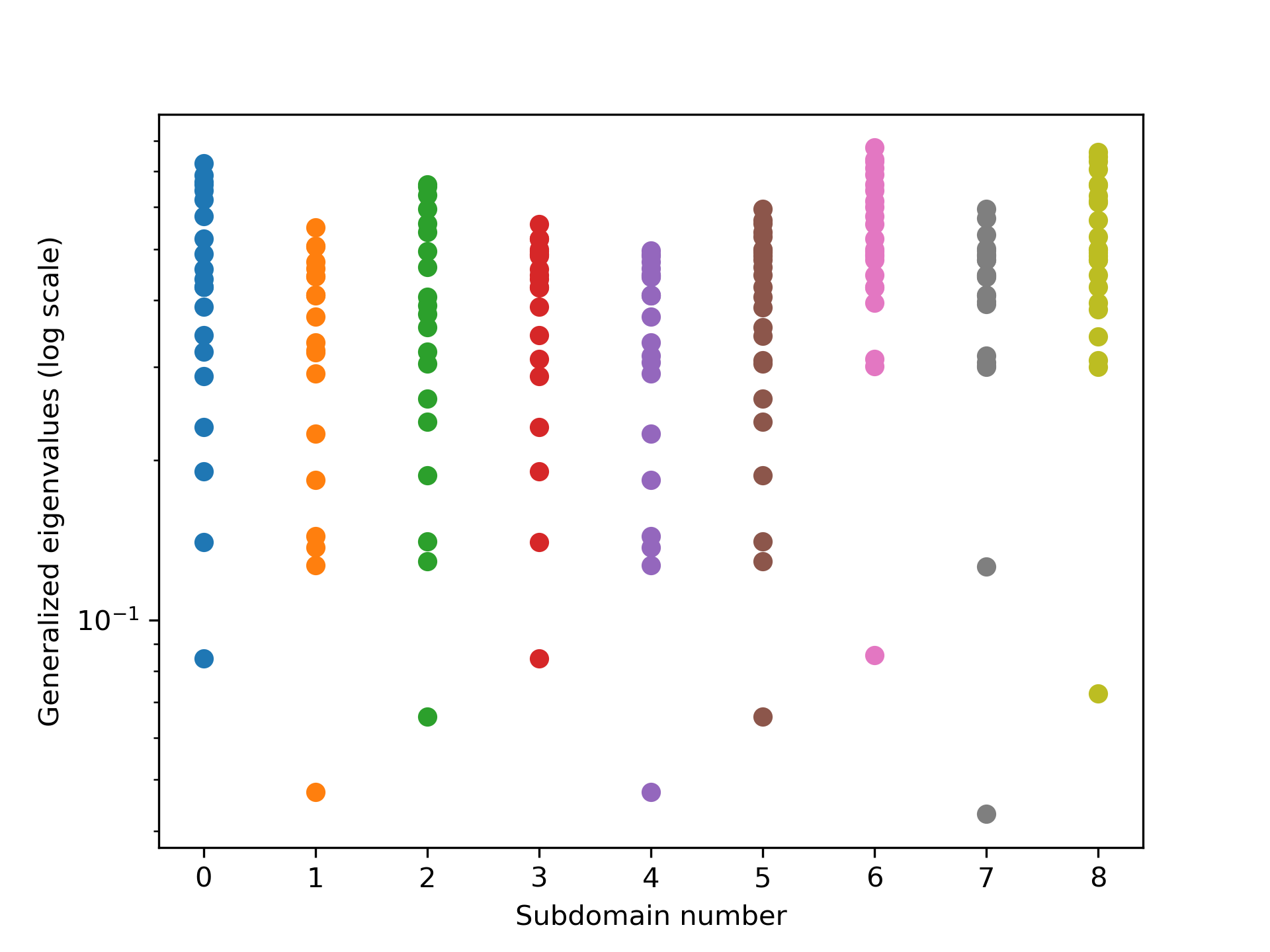}
\end{center}
\caption{For $\bH_2 = \bH_{\mathrm{hyb}}\NN(\tau_\sharp)$: 20 smallest non-zero eigenvalues of the GenEO eigenproblem in each subdomain}
\label{fig:PCNewNN-eigenvalues}
\end{figure}

\begin{figure}
\begin{center}
\includegraphics[width=0.8\textwidth]{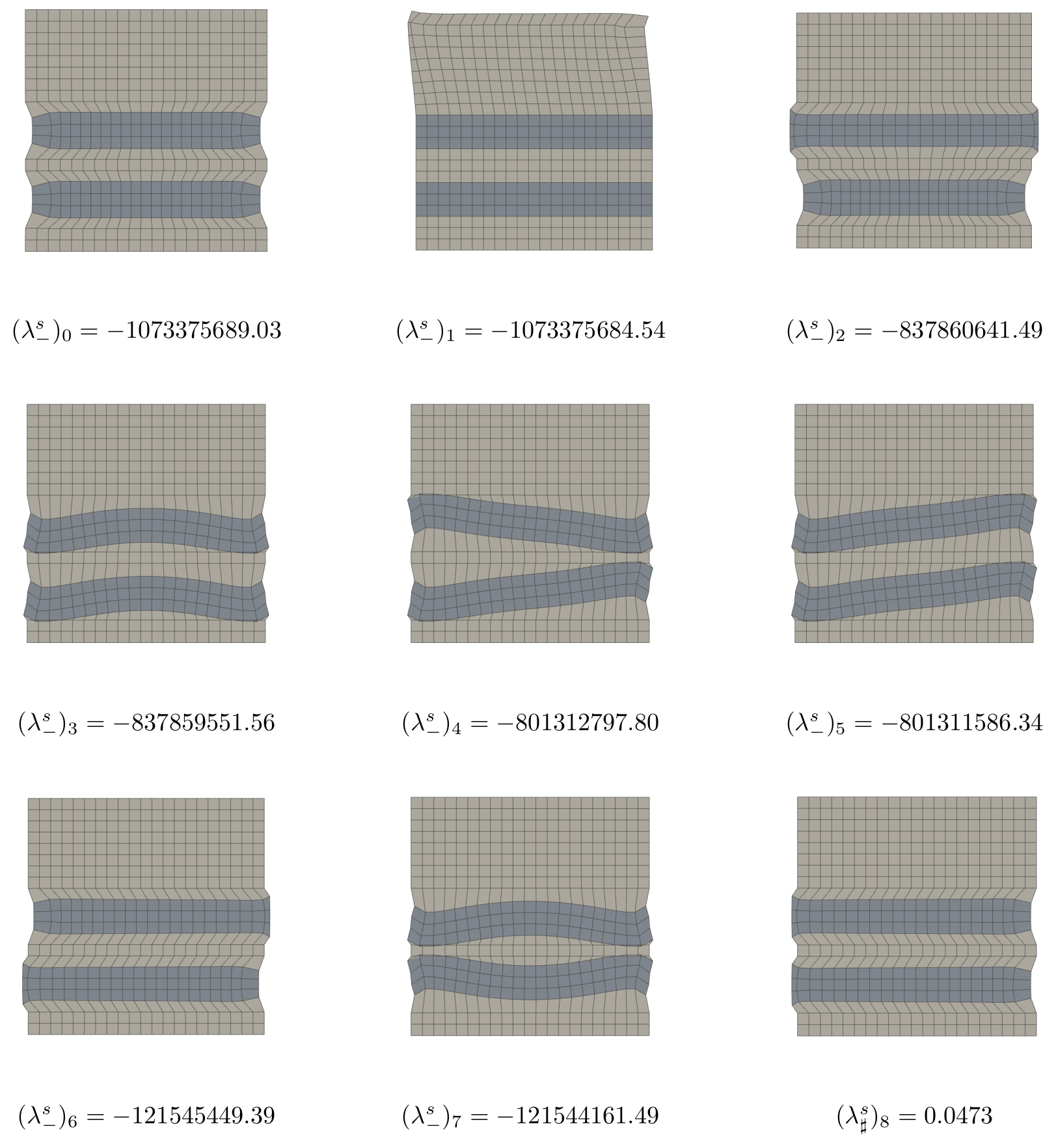}
\end{center}
\caption{For $\bH_2 = \bH_{\mathrm{hyb}}\NN(\tau_\sharp)$: plot of the 9 vectors that are selected for the coarse space in the central subdomain ($s=4$) represented as deformations. The colors correspond to the two values of $E$ (darker color is harder material). The first eight vectors correspond to a zero eigenvalue in the GenEO eigenproblem, \textit{i.e.}, to a negative eigenvalue $\lambda\s_-$ of $\bB\s$. The last vector corresponds to the first non-zero eigenvalue $(\lambda^s_\sharp)_8$ in the GenEO eigenproblem.}
\label{fig:PCNewNN-CV4}
\end{figure}

In all that follows we focus on the choice of preconditioner: 
\[
\bH_{3,\mathrm{ad}} \text{ with } \bH_2 = \bH_{\mathrm{hyb}}\NN(\tau_\sharp).
\]
With this choice of $\bH_2$, the theory predicts that the smallest eigenvalue of the preconditioned operator is larger than $1$ and this bound is observed to be sharp in Table~\ref{tab:allprecs} (and more generally throughout our numerical experiments). For this reason, we no longer report on the extreme eigenvalues. Instead we only give values of the condition number $\kappa$.

\paragraph{Influence of the threshold $\tau_\sharp$}

For this test we study the influence of $\tau_\sharp$. When $\tau_\sharp$ increases, more vectors are selected for the coarse space and the condition number bound decreases. AWG is compared to classical Neumann Neumann GenEO (which is not algebraic). Figure~\ref{fig:kappavsdimV0} is a plot of the condition number of the preconditioned operator versus the size of the GenEO coarse space. Recall that GenEO has the disadvantage of not being algebraic but the AWG method has the disadvantage of having a second coarse space (of size $48$). For AWG, the size of the coarse space cannot go below $n_-$ because the kernels of $\bA\s_-$ are always selected. The study was performed by running the simulations for $\tau_\sharp \in [0, 0.001, 0.01, 0.05, 0.1, 0.2, 0.5]$. Two values of Poisson's ratio $\nu$ are considered $\nu = 0.3$ and $\nu = 0.4$. For $\nu = 0.3$ the AWG coarse space to achieve a condition number of 10 is almost the same as with classical GenEO. When $\nu = 0.4$ more vectors are required for AWG. 

\begin{figure}
\begin{center}
\begin{tabular}{cc}
$\nu = 0.3$ & $\nu = 0.4$ \\
\includegraphics[width=0.49\textwidth]{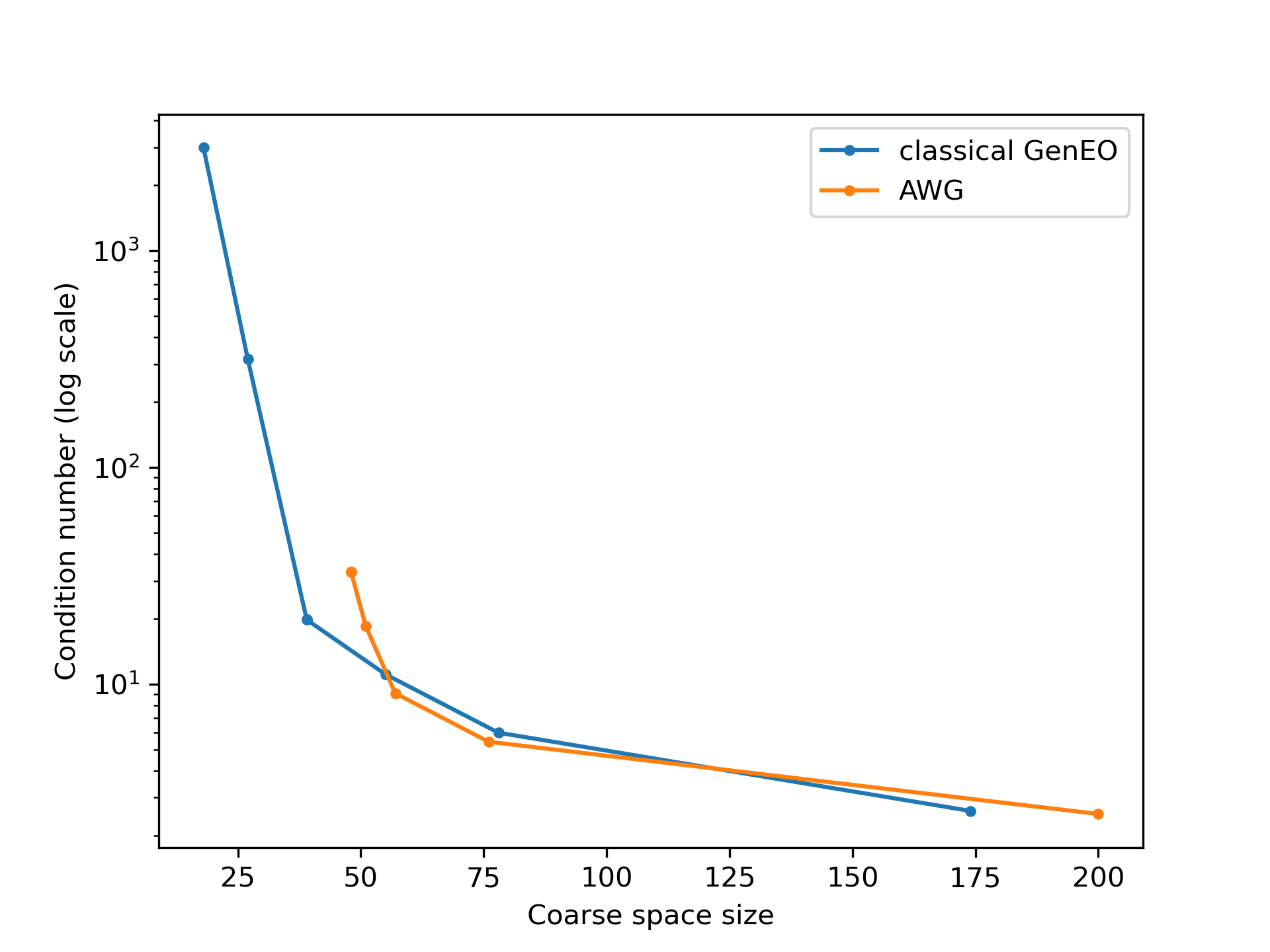}
&
\includegraphics[width=0.49\textwidth]{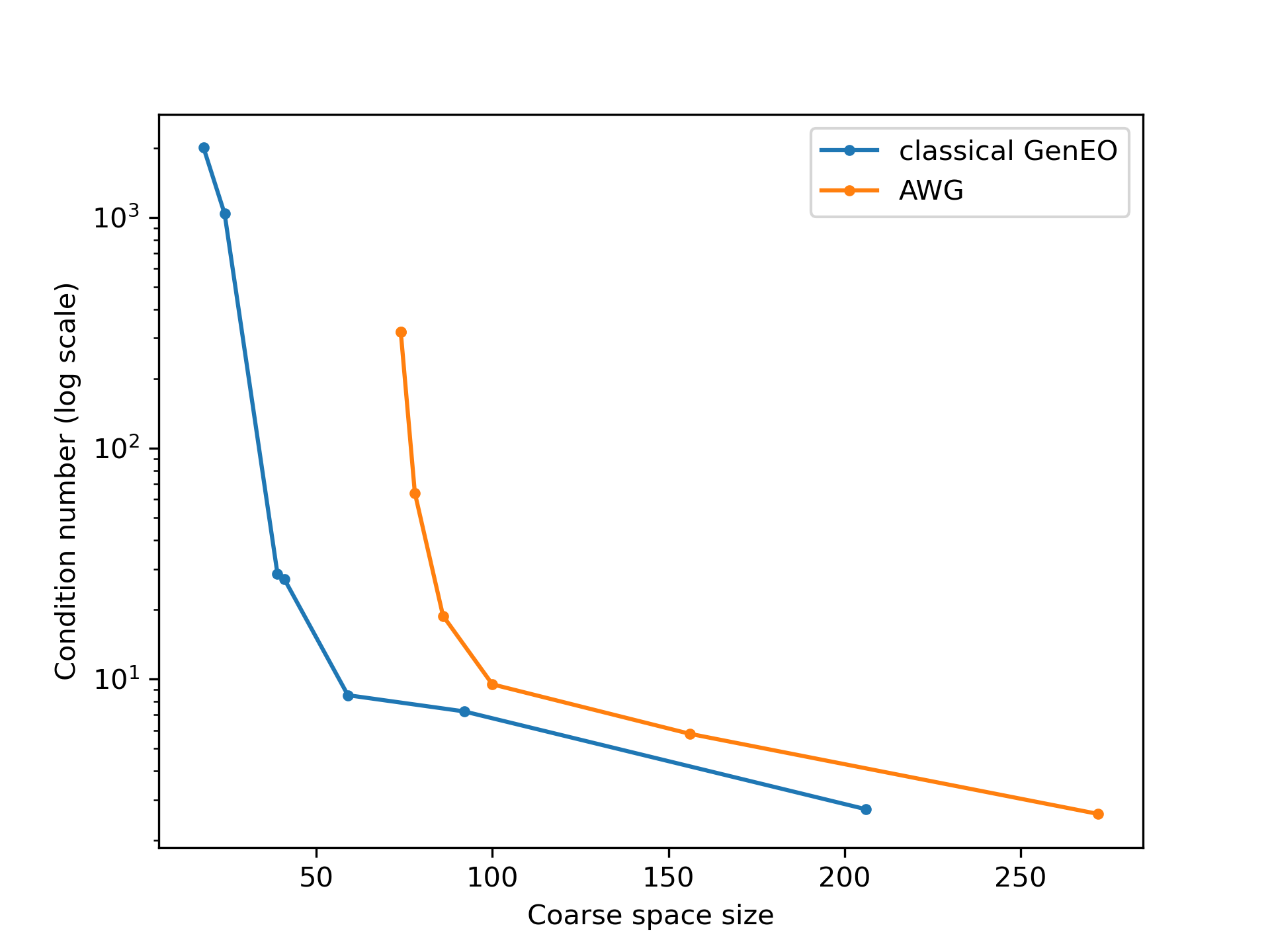}
\end{tabular}
\end{center}
\caption{\label{fig:kappavsdimV0} Condition number with respect to coarse space size for $\bH_{3,\mathrm{ad}}$ with $\bH_2 = \bH_{\mathrm{hyb}}\NN(\tau_\sharp)$ and the comparable classical GenEO coarse space. $\tau_\sharp \in [0, 0.001, 0.01, 0.05, 0.1, 0.2, 0.5]$. Left: $\nu = 0.3$. Right: $\nu = 0.4$.}
\end{figure}

\paragraph{Influence of Poisson's ratio $\nu$}

For this study, the value of $\nu$ varies between $0.2$ and $0.49$. Young's modulus is kept constant in the domain and equal to $E = 10^{11}$. The threshold is $\tau_\sharp = 0.05$ (so slightly smaller than previously). As shown in Table~\ref{tab:case2-nuvaries}, increasing $\nu$ has quite a dramatic effect on $n_-$ (even without going near the incompressible limit). This is rather disappointing. The classical (non algebraic) GenEO does not suffer from this problem (away from the incompressible limit $\nu \rightarrow 0.5$).

\begin{table}
\centering
\begin{tabular}{lrrrrr}
\multicolumn{5}{c}{AWG}\\
\hline
$ \nu$ &  $\kappa$ &  $It$ &  $\#V_0$ &  $n_- $\\
\hline
 0.20 &  19.7 &    33 &     21 &                 12 \\
 0.30 &  20.3 &    32 &     29 &                 19 \\
 0.35 &  18.6 &    32 &     47 &                 25 \\
 0.40 &  25.8 &    39 &     98 &                 70 \\
 0.45 &  27.1 &    29 &    115 &                110 \\
 0.49 &  16.8 &    25 &    362 &                357 \\
\hline
\end{tabular}
\qquad
\begin{tabular}{lrrrrr}
\multicolumn{5}{c}{Classical GenEO}\\
\hline
$ \nu$ &  $\kappa$ &  $It$ &  $\#V_0$ &  $n_- $\\
\hline
  0.20 &  17.2 &    33 &     21 &   0 \\
  0.30 &  17.6 &    36 &     21 &   0 \\
  0.35 &  19.1 &    37 &     21 &   0 \\
  0.40 &  20.1 &    39 &     24 &   0 \\
  0.45 &  33.7 &    46 &     28 &   0 \\
  0.49 &  34.9 &    51 &     94 &   0 \\
\hline
\end{tabular}
\caption{\label{tab:case2-nuvaries}The influence of Poisson's ratio $\nu$ is studied when $E$ is constant and equal to $10^{11}$. The threshold is $\tau_\sharp = 0.05$. $\nu$: Poisson's ratio, $\kappa$: condition number of preconditioned operator, $It$: number of iterations, $\#V_0$: dimension of GenEO coarse space,  $n_- = \operatorname{rank}(\bA_-)$: dimension of second coarse space. Recall that classical GenEO is not algebraic.}
\end{table}

\paragraph{Influence of $E$}
The threshold is set back to $\tau_\sharp = 0.1$ and Poisson's ratio to $\nu = 0.3$ in all that follows. This time, the values $E_1$ and $E_2$ of Young's modulus in, respectively, the dark and light parts of $\omega$ in Figure~\ref{fig:E-solution} are varied. The results are in Table~\ref{tab:case3-Evaries}. We observe that all AWG condition numbers are between $8$ and $12.2$ so they are all very small and fast convergence is guaranteed. The smallest coarse space size (both for the GenEO coarse space and the second coarse space) is for the case where $E$ is constant throughout $\omega$. The cases where $E_1 > E_2$ (hard layers in softer material) require smaller coarse spaces than the cases where $E_2 > E_1$ (soft layers in harder material). Finally, the AWG coarse spaces are always larger than the (non algebraic) GenEO coarse spaces but not significantly in five cases out of seven.

\begin{table}
\centering
\begin{tabular}{lrrrr}
\multicolumn{5}{c}{AWG}\\
\hline
 $(E_1,E_2)$  &  $\kappa$ &  $It$ &  $\#V_0$ &  $n_- $\\
\hline
$(10^{5} ,  10^{11})$ &  10.8 &    22 &     95 &                 75 \\
$(10^{7} ,  10^{11})$ &  10.8 &    23 &     95 &                 75 \\
$(10^{9} ,  10^{11})$ &  10.4 &    24 &     94 &                 73 \\
$(10^{11} ,  10^{11})$ &  12.2 &    29 &     35 &                 19 \\
$(10^{11} ,  10^{9})$ &   8.0 &    26 &     59 &                 48 \\
$(10^{11} ,  10^{7})$ &   9.0 &    26 &     57 &                 48 \\
$(10^{11} ,  10^{5})$ &   8.4 &    29 &     57 &                 48 \\
\hline
\end{tabular}
\qquad
\begin{tabular}{lrrrr}
\multicolumn{5}{c}{Classical GenEO}\\
\hline
 $(E_1,E_2)$ &  $\kappa$ &  $It$ &  $\#V_0$ &  $n_- $\\
\hline
 $(10^{5} ,  10^{11})$ &   8.6 &    23 &     90 &                  0 \\
 $(10^{7} ,  10^{11})$ &   8.6 &    26 &     87 &                  0 \\
 $(10^{9} ,  10^{11})$ &   8.5 &    25 &     85 &                  0 \\
 $(10^{11} ,  10^{11})$ &  13.7 &    32 &     28 &                  0 \\
 $(10^{11} ,  10^{9})$ &  11.2 &    30 &     52 &                  0 \\
 $(10^{11} ,  10^{7})$ &  11.1&    29 &     55 &                  0 \\
 $(10^{11} ,  10^{5})$ &  12.7 &    30 &     55 &                  0 \\
\hline
\end{tabular}
\caption{\label{tab:case3-Evaries}The influence of $E$ and of the jump between $E_1$ and $E_2$ is studied. $(E_1, E_2)$: values of Young's modulus in the layers of coefficients, $\kappa$: condition number of preconditioned operator, $It$: number of iterations, $\#V_0$: dimension of GenEO coarse space,  $n_- = \operatorname{rank}(\bA_-)$: dimension of second coarse space. Recall that classical GenEO is not algebraic.} 
\end{table}

\paragraph{Influence of the accuracy of $\bW$} 

The second coarse space for the AWG preconditioners is computed by solving $n_-$ linear systems: $\bA_+ \backslash( {\bR\s}^\top\bv\s_-)$ for the vectors $\bv\s_-$ that correspond to a negative eigenvalue of $\bB\s$ ($s \in \llbracket 1 , N \rrbracket$). Until now, we have solved these with very high accuracy: the relative residual tolerance $rtol$ was set to $10^{-10}$. In this experiment we vary $rtol$. Table~\ref{tab:case5and8-rtolvaries} shows how increasing $rtol$ affects the condition number of $\bA$ preconditioned by AWG. Two cases have been studied with different Poisson's ratios: $\nu = 0.3$ and $\nu = 0.4$. Up to $rtol = 10^{-2}$ there is no change compared to $rtol = 10^{-10}$ and these intermediary results have not been reported in the table. In fact, up to $rtol = 0.5$, the condition number is hardly degraded. It must be kept in mind that this is a case with few subdomains and what we observe to be a small change in $\kappa$ could become more significant with more subdomains. Still, the conclusion is optimistic: the linear solves with $\bA_+$ do not need to be overly precise.

\begin{table}
\centering
\begin{tabular}{lrrrr}
\multicolumn{5}{c}{$\nu=0.3$}\\
\hline
 $rtol$  &  $\kappa$ &  $It$ &  $\#V_0$ &  $n_- $\\
\hline
  $10^{-10}$ &     9.0 &    26 &      57 &                 48 \\
  $10^{-2}$ &     9.0 &    27 &      57 &                 48 \\
  $0.05$    &    11.1 &    31 &      57 &                 48 \\
  $0.1$     &    12.2 &    32 &      57 &                 48 \\
  $0.5$     &   400.8 &    40 &      57 &                 48 \\
  $0.9$     &   706.8 &    64 &      57 &                 48 \\
\hline
\end{tabular}
\qquad
\begin{tabular}{lrrrr}
\multicolumn{5}{c}{$\nu=0.4$}\\
\hline
 $rtol$  &  $\kappa$ &  $It$ &  $\#V_0$ &  $n_- $\\
\hline
  $10^{-10}$ &      9.4 &    29 &     100 &                 74 \\
  $10^{-2}$ &       9.4 &    30 &     100 &                 74 \\
  $0.05$    &      12.0 &    33 &     100 &                 74 \\
  $0.1$     &      17.4 &    36 &     100 &                 74 \\
  $0.5$     &    1563.3 &    88 &     100 &                 74 \\
  $0.9$     &    2142.1 &   100 &     100 &                 74 \\
\hline
\end{tabular}
\caption{\label{tab:case5and8-rtolvaries}Influence of the accuracy of $rtol$ up to which the linear systems with $\bA_+$ preconditioned by $\bH_2$ are solved during the setup of the second coarse basis $\bW$. $rtol$: tolerance, $\kappa$: condition number of preconditioned operator, $It$: number of iterations, $\#V_0$: dimension of GenEO coarse space,  $n_- = \operatorname{rank}(\bA_-)$: dimension of second coarse space. Recall that classical GenEO is not algebraic.} 
\end{table}

\paragraph{Varying number of harder layers}

This time it is the number of layers of the harder coefficient that varies. The case with six layers is the usual one from \eqref{eq:distribE} represented in Figure~\ref{fig:E-solution}. The case with nine layers is obtained by also setting $E=10^{11}$ if $ (\operatorname{floor}(y) - y)  \in [5/7, 6/7] $. The case with three layers is obtained by setting $E=10^{11}$ only if $ (\operatorname{floor}(y) - y)  \in [1/7,2/7] $. Two additional cases with homogeneous hard and soft material are also considered. The results are shown in table~\ref{tab:case6-nbstripevaries}. 

We first observe that distributions of $E$ with more discontinuities require more coarse vectors. The AWG coarse space is always larger than the (non algebraic) classical GenEO coarse space but not significantly. We also check that homogeneous distributions of $E$ have the same behaviour for different values of $E$ and this is expected because $\bA$ and $\bb$ are linear in $E$.
 
\begin{table}
\centering
\begin{tabular}{lrrrr}
\multicolumn{5}{c}{AWG}\\
\hline
 $E$ &  $\kappa$ &  $It$ &  $\#V_0$ &  $n_- $\\
\hline
   $E=10^{11}$ &  12.2 &    29 &     35 &                 19 \\
   9 layers&   4.9 &    17 &     72 &                 72 \\
   6 layers&   9.0 &    26 &     57 &                 48 \\ 
   3 layers &   9.8 &    29 &     43 &                 25 \\
   $E=10^7$   &  12.2 &    29 &     35 &                 19 \\
\hline
\end{tabular}
\qquad
\begin{tabular}{lrrrr}
\multicolumn{5}{c}{Classical GenEO}\\
\hline
 $E$ &  $\kappa$ &  $It$ &  $\#V_0$ &  $n_- $\\
\hline
   $E=10^{11}$   &  13.7 &    32 &     28 &                  0 \\
   9 layers & 4.8 &    20 &       69 &                  0 \\
   6 layers &  11.1 &    29 &     55 &                  0 \\
   3 layer &   9.9 &    31 &     35 &                  0 \\
   $E=10^7$ &  13.7 &    32 &     28 &                  0 \\
\hline
\end{tabular}
\caption{\label{tab:case6-nbstripevaries}The number of layers of the harder coefficient varies. Two cases with homogeneous $E$ are also considered. $\kappa$: condition number of preconditioned operator, $It$: number of iterations, $\#V_0$: dimension of GenEO coarse space,  $n_- = \operatorname{rank}(\bA_-)$: dimension of second coarse space. Recall that classical GenEO is not algebraic.} 
\end{table}

\paragraph{Long domain with layers of coefficients}

\begin{figure}
\begin{center}
\includegraphics[width=0.7\textwidth, trim = 5cm 9cm 5cm 9cm, clip]{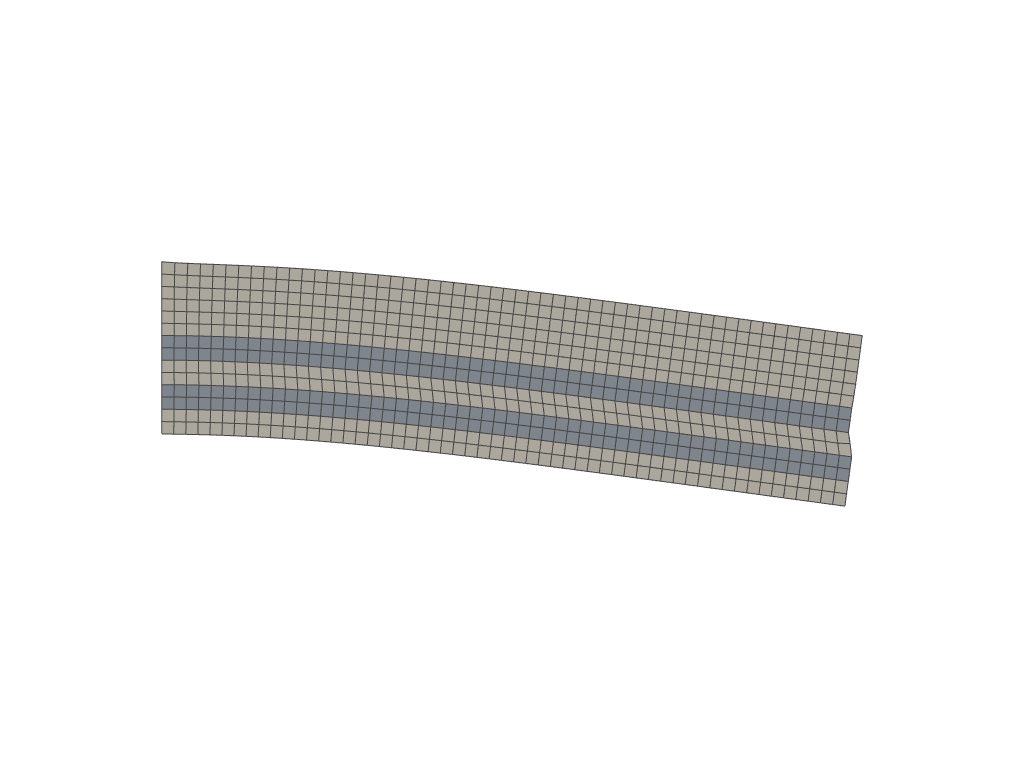} 
\end{center}
\caption{\label{fig:barre-E-solution}The computational domain $\omega$ has been displaced by $10^5$ multiplied by the solution $\bx_*$. The colors show the distribution of coefficients. Dark: $E = E_1 = 10^{11}$ - Light :  $E = E_2 = 10^{7}$.}
\end{figure}

This test case (represented in Figure~\ref{fig:barre-E-solution}) is often studied in domain decomposition articles and presentations. Its drawback is that it doesn't have any crosspoints (degrees of freedom that belong to more than two subdomains) but all simulations up until now had crosspoints and they don't appear to be an issue for AWG preconditioners. In this paragraph, weak scalability is studied. The parameters in the test case are: 
\begin{itemize}
\item $N \in \{2;4;8;15;29\}$ (number of unit-square subdomains),
\item $\omega = [N,1]$ (computational domain parametrized by number of subdomains), 
\item $h = 1/14$ (mesh size),
\item $\nu = 0.3$ (Poisson's ratio),
\item 
$
E = \left\{\begin{array}{l}
E_1 = 10^{11} \text{ if } y \in [1/7;2/7] \cup [3/7; 4/7], 
\\E_2 = 10^7 \text{ otherwise},
\end{array}
 \right. 
$
(Young's modulus),
\item $\bH_{3,\mathrm{ad}} \text{ with } \bH_2 = \bH_{\mathrm{hyb}}\NN(0.1)$ (preconditioner),
\item $rtol=10^{-10}$ (relative residual tolerance for the linear solves with $\bA_+$ and the linear solve with $\bA$). 
\end{itemize}

The results are shown in Table~\ref{tab:case7-Nvaries}. As predicted theoretically, with the AWG preconditioner the condition number hardly increases with the number of subdomains and this points towards weak scalability. For all values of $N$, the subdomains are identical (with a difference between ones that are at the edge of $\omega$ and others) so the coarse space grows almost linearly with the number of subdomains. This can be viewed as a first weak scalability result. 

\begin{table}
\centering
\begin{tabular}{lrrrr}
\multicolumn{5}{c}{AWG}\\
\hline
 $N$ &  $\kappa$ &  $It$ &  $\#V_0$ &  $n_- $\\
\hline
  2 &  12.6 &    15 &      8 &     8 \\
  4 &   9.8 &    16 &     26 &    20 \\
  8 &   9.0 &    15 &     62 &    44 \\
 15 &   8.8 &    15 &    125 &    86 \\
 29 &   8.7 &    17 &    251 &   170 \\

\hline
\end{tabular}
\qquad
\begin{tabular}{lrrrr}
\multicolumn{5}{c}{Classical GenEO}\\
\hline
 $N$ &  $\kappa$ &  $It$ &  $\#V_0$ &  $n_- $\\
\hline
   2 &   9.5 &  15 &      7 &  0 \\
   4 &  11.9 &  19 &     19 &  0 \\
   8 &  12.6 &  23 &     43 &  0 \\
  15 &  12.8 &  27 &     85 &  0 \\
  29 &  12.8 &  28 &    169 &  0 \\
\hline
\end{tabular}
\caption{\label{tab:case7-Nvaries}The number $N$ of subdomains increases, the problem size is proportional to the number of subdomains. Weak scalable behaviour would be for the number of iterations to remain constant and this is what is observed. $N$: number of subdomains, $\kappa$: condition number of preconditioned operator, $It$: number of iterations, $\#V_0$: dimension of GenEO coarse space,  $n_- = \operatorname{rank}(\bA_-)$: dimension of second coarse space. Recall that classical GenEO is not algebraic.}
\end{table}

\section{Conclusion}
In this article new preconditioners called AWG for Algebraic Woodbury-GenEO have been introduced. Combined with PCG, they are algebraic domain decomposition methods with two coarse spaces. Convergence in a small number of iterations can be guaranteed by adjusting some user chosen thresholds and enriching one of the coarse spaces accordingly. Numerical results have been presented as a proof of concept and to illustrate the behaviour of the AWG preconditioners on some simple test cases. Further numerical simulations must be performed to assess the overall efficiency of the AWG preconditioners. Some possible improvements of the AWG preconditioners are still under investigation. This includes decreasing the size $n_-$ of the second coarse space by proposing other choices of $\bB\s$, finding sparser approximations for the second coarse space and coarse solve, and perhaps, injecting some information into the preconditioner like the near kernel of $\bA$ in a way inspired by smoothed aggregation multigrid \cite{van2001convergence}. 

\bibliographystyle{abbrv}
\bibliography{AlgebraicGenEO}
\end{document}